\documentclass[a4paper, oneside, reqno]{amsart}

\allowdisplaybreaks[1]

\usepackage{amssymb}
\usepackage{xfrac}
\usepackage{mathtools} 
\usepackage[shortlabels]{enumitem} 
\usepackage{soul}
\usepackage[largesc]{newpxtext} 
\usepackage[vvarbb]{newpxmath}  
\usepackage{array}
    \newcolumntype{C}{>{$}c<{$}}
    \newcolumntype{L}{>{$}l<{$}}
    \newcolumntype{R}{>{$}r<{$}}

\usepackage{pgfplots}
    \pgfplotsset{compat=1.16}
\usepackage{tikz}
    \tikzstyle{point} = [circle, draw, fill=blue!80, 
    text width=0.5em, text centered, 
    node distance=3cm, inner sep=0pt]
    \tikzstyle{line} = [draw, -]
\usepackage{graphicx}

\usepackage{standalone} 
\usepackage{tikz-3dplot} 

\usepackage{xcolor}
\definecolor{myurlcolor}{rgb}{0.1,0.1,0.8}
\definecolor{mylinkcolor}{rgb}{0.05,0.05,0.4}

 \usepackage{hyperref}
\hypersetup{colorlinks,
	linkcolor=mylinkcolor,
	citecolor=mylinkcolor,
	urlcolor=myurlcolor}

\usepackage{cleveref}

\newtheorem{theorem}{Theorem}[section]
\newtheorem{lemma}[theorem]{Lemma}
\newtheorem{prop}[theorem]{Proposition}
\newtheorem{cor}[theorem]{Corollary}
\newtheorem{conj}[theorem]{Conjecture}

\newtheorem*{thmspecial-existence}{Corollary~\ref{cor:special-existence}}
\newtheorem*{thmstrict_negative_type_microwtg}{Theorem~\ref{thm:strict_negative_type_microwtg}}
\newtheorem*{thmneg_type_schoenberg_sphere}{Theorem~\ref{thm:neg_type_schoenberg_sphere}}

\newtheorem*{thmmicrowtg_consequences}
{Theorem~\ref{thm:microwtg_consequences}}
\newtheorem*{thmmicrowtg_energy_maximizing}
{Theorem~\ref{thm:microwtg_energy_maximizing}}
\newtheorem*{thmmicrowtg-curvature}
{Theorem~\ref{thm:microwtg-curvature}}
\newtheorem*{thmgeneral-existence}
{Theorem~\ref{thm:general-existence}}
\newtheorem*{propUL-inverse}
{Proposition~\ref{prop:UL-inverse}}

\theoremstyle{definition}
\newtheorem{definition}[theorem]{Definition}
\newtheorem{example}[theorem]{Example}
\newtheorem{remark}[theorem]{Remark}

\theoremstyle{remark}

\numberwithin{equation}{section}

\newcommand{\demph}[1]{\emph{#1}}

\DeclareMathOperator{\adj}{adj}

\renewcommand{\d}{d}
\renewcommand{\vec}{\mathbf}
\renewcommand{\th}{\text{th}}

\DeclarePairedDelimiterXPP{\matmag}[1]{\operatorname{mag}}{(}{)}{}{#1}

\newcommand{\magfunname}[1]{\mu_{#1}}
\newcommand{\magderivname}[1]{\mu'_{#1}}
\DeclarePairedDelimiterXPP{\magfun}[2]{\magfunname{#1}}{(}{)}{}{#2}
\DeclarePairedDelimiterXPP{\magderiv}[2]{\magderivname{#1}}{(}{)}{}{#2}
\DeclarePairedDelimiterXPP{\con}[1]{\operatorname{con}}{(}{)}{}{#1}
\newcommand{\onevector}{\mathbf{1}}
\newcommand{\onematrix}{\onevector \onevector^\transpose}
\newcommand{\zerovector}{\mathbf{0}}
\newcommand{\transpose}{\mathrm{T}}
\newcommand{\pseudo}{\dagger}

\DeclareMathOperator{\im}{im}

\newcommand{\R}{\mathbb{R}}
\newcommand{\define}{\textbf}
\newcommand{\dd}{\mathrm{d}} 
\DeclareMathOperator{\lagrangian}{\mathcal{L}}
\newcommand{\micro}{\widehat{\vec{w}}}

\newcommand{\finsuppmeas}{\mathrm{FM}}
\newcommand{\wtspace}{\mathcal{W}}
\newcommand{\dispair}[2]{\langle #2,#1 \rangle}

\newcommand{\proj}{\pi} 
\newcommand{\inc}{\iota} 
\newcommand{\dproj}{\bar\pi} 
\newcommand{\dinc}{\bar\iota} 
\newcommand{\projm}{\proj} 
\newcommand{\incm}{\inc} 
\newcommand{\changem}{\begin{bmatrix}\projm_1\\\projm_0\end{bmatrix}} 
\newcommand{\changeminv}{\begin{bmatrix}\incm_1 \incm_0\end{bmatrix}} 

\makeatletter
\newsavebox{\@bra}
\newsavebox{\@brb}
\DeclarePairedDelimiterX{\Zpair}[2]{
  \delimsize\langle%
  \hspace*{0.3mm}\hspace*{0.75mm}\savebox{\@bra}{\(\displaystyle\left\langle\vphantom{#1}\right.\)}\hspace*{-1.035\wd\@bra}%
  \delimsize\langle%
}{
  \delimsize\rangle%
  \hspace*{0.3mm}\hspace*{0.75mm}\savebox{\@brb}{\(\displaystyle\left.\vphantom{#1}\right\rangle\)}\hspace*{-1.035\wd\@brb}%
  \delimsize\rangle
}
{#1, #2}
\makeatother

\DeclareMathOperator{\measmassone}{\mathcal{M}_1}
\newcommand{\measmasszero}[1]{\mathcal{M}_0(#1)}
\newcommand{\N}{\mathbb{N}}
\newcommand{\isomto}{\xrightarrow{\raisebox{-0.5ex}[0ex][0ex]{$\sim$}}}
\DeclareMathOperator{\besselspace}{H}

\begin{document}

\title{The microscopic weighting on a metric space}

\author{Emily Roff}
\address{ER: School of Mathematics and the Maxwell Institute for Mathematical Sciences, University of Edinburgh, Scotland}

\author{Simon Willerton}
\address{SW: School of Mathematical and Physical Sciences, University of Sheffield, England}


\date{November 2, 2026}

\begin{abstract}
We introduce the \emph{microscopic weighting}, a canonical signed measure of mass one that can be associated to almost any finite metric space. The microscopic weighting is obtained as the small-scale limit of the weightings used to define the magnitude function. We give general criteria for its existence, proving in particular that every finite space of strictly negative type admits a microscopic weighting; this includes every finite subset of Euclidean or hyperbolic space and every finite tree. Heuristically speaking, the microscopic weighting distributes its mass as widely as possible across a space, assigning greater weight to sparse or outlying regions and emphasizing points on the boundary. Indeed, we show that on a finite space of negative type the microscopic weighting (when it exists) offers a canonical choice of optimizing measure for an energy integral determined by the distance function. Alternatively, it can be characterized in terms of the geometry of the Schoenberg embedding. Each of these interpretations also clarifies the information carried by the derivative of the magnitude function at zero. Though our main focus in this paper is on finite metric spaces, we lay the groundwork to extend the theory to compact subsets of Euclidean space. In that setting, we observe that the microscopic weighting must be understood as a distribution rather than as a measure.
\end{abstract}

\maketitle

\setcounter{tocdepth}{1}
\tableofcontents


\section{Introduction}
\label{sec:introduction}

This paper introduces and investigates a novel isometric invariant of finite metric spaces: the \demph{microscopic weighting}. The microscopic weighting is a canonical signed measure of mass one that can be attached to almost any finite metric space, including any finite subset of Euclidean or hyperbolic space, and any finite tree. It is defined as the small-scale limit of a family of scale-sensitive signed measures, known as weightings, that play a central role in the far-reaching theory of magnitude.

For a large class of spaces---those of negative type---the microscopic weighting can be given at least two complementary geometric interpretations. The first has the flavour of classical distance geometry: it relates the microscopic weighting on a space \(X\) to the shape of a convex polytope associated to \(X\) by a theorem of Schoenberg. The second interpretation is more intrinsic to the geometry of \(X\) itself: it says that the microscopic weighting is a canonical choice of maximizing measure for an energy integral determined by the distance function. Thus, the microscopic weighting on \(X\) can be interpreted as a signed measure of mass one that is spread out as widely as possible across \(X\). Moreover, for a finite subset of Euclidean space---or any finite space of strictly negative type---the maximal value of the energy integral is equal to the derivative at zero of the magnitude function.

Though our primary focus in this paper is on finite metric spaces, we will also begin to explore the extension of these ideas to compact spaces. We will see that for a compact subset of Euclidean space, a microscopic weighting must be understood not as a measure but as a distribution, in the sense of Schwartz. Focusing on the example of the three-dimensional unit ball, we will make a series of observations that lead us to conjecture that the relationship between maximal energy and magnitude extends from finite to compact spaces.

In this introduction we begin by describing the magnitude function of a finite metric space and the weightings used to compute it. That leads us to the definition of the microscopic weighting, after which we outline the main results of the paper.


\subsection{The magnitude function}
\label{subsec:magnitude-function}

The study of magnitude has its roots in category theory, but its branches reach out into geometry, topology, analysis, and combinatorics. The idea is that to every finite metric space \(X\) one can assign a function \(\mu_X \colon (0, \infty) \to \R \cup \{\infty\}\), whose parameter should be thought of as controlling the \emph{scale} of the metric on \(X\), and whose values record the \emph{effective size} of \(X\) as the scale varies. The function \(\mu_X\) is defined by Leinster in \cite{LeinsterMagnitude2013} as follows.

First, given a metric space \(X\) with points \(x_1,\ldots, x_n\) and distance function \(d\), for each \(t > 0\) let \(tX\) denote the space with the same points, in which the distance from \(x_i\) to \(x_j\) is \(t d(x_i,x_j)\). We will refer to \(tX\) as being the space \(X\) \define{at scale} \(t\). The \define{similarity matrix} of \(X\) at scale \(t\) is the \(n \times n\) matrix \(Z(t)\) with entries \(Z(t)_{ij} = \exp(-td(x_i,x_j))\).

To compute \(\mu_X(t)\), one looks for a vector \(\vec{w}(t) \in \R^n\) satisfying 
\begin{equation}
\label{eq:weight_equation}
    Z(t) \vec{w}(t) = \begin{pmatrix} 1 & 1 & \cdots & 1 \end{pmatrix}^\transpose.
\end{equation}
Such a vector is said to be a \define{weighting} on \(X\) at scale \(t\); we will think of it as a signed measure on \(X\), with \(w_i(t) \coloneq (\vec{w}(t))_i\) denoting the measure of the point \(x_i\). 
If \(X\) has a weighting at scale \(t\), then the \define{magnitude of \(X\)} at scale \(t\) is defined by
\[\mu_X(t) \coloneq \sum_{i=1}^n w_i(t)\]
for any weighting \(\vec{w}(t)\). If more than one weighting exists, this sum is independent of the weighting chosen to compute it (see \Cref{sec:concentration_of_matrices}). If no weighting exists at scale \(t\), we will say that \(\mu_X(t) = \infty\). The \define{magnitude function} of \(X\) is the function
\begin{align*}
\mu_X \colon (0, \infty) &\to \R \cup \{\infty\} \\
t &\mapsto \mu_X(t).
\end{align*}
This definition can be extended from finite metric spaces to many compact metric spaces \cite[\S 3]{LeinsterMagnitude2013}; we give some details in \Cref{sec:infinite-spaces}.

Though it is not obvious from the definition, the magnitude function is actually an instance of a very general notion of Euler characteristic \cite[\S 1.3]{LeinsterMagnitude2013}. Accordingly, magnitude shares various formal properties with topological Euler characteristic and with its more familiar specialization, the cardinality of finite sets. 

Crucially, though, \(\mu_X(t)\) varies with \(t\): unlike cardinality or topological Euler characteristic, magnitude is sensitive to scale. This scale-sensitivity is precisely what makes the magnitude function interesting. For a finite metric space \(X\), one can often understand \(\mu_X\) as recording the \emph{effective number of points} in \(X\) as the scale of the metric varies~\cite[Example~6.4.6]{LeinsterEntropy2021} and its instantaneous rate of growth as recording the \emph{effective dimension}~\cite[\S 4]{WillertonSpread2015}.


\subsection{Microscopic weightings}
\label{subsec:microweightings}

By definition, a weighting \(\vec{w}(t)\) on \(X\) describes the contribution of each point in \(X\) to its magnitude at scale \(t\). In typical cases the weights are unevenly distributed, and their distribution carries finer information about the geometry of the space. Heuristically speaking, points with many near neighbours tend to make a smaller contribution to magnitude than those with fewer near neighbours: this reflects the intuition that a tight cluster of several points is `effectively' a single point within the space. Accordingly, a weighting on \(X\) will tend to distribute more of its mass towards sparser or outlying regions.

Indeed, it is demonstrated by Bunch \textit{et al} in \cite{Bunch-etal} that weightings can be used to detect the boundary of a finite set approximating a region of Euclidean space, leading to their implementation as a method of edge detection in the context of computer vision (Adamer \textit{et al}, \cite{ABOR2019}). In a very different setting, within theoretical ecology, weightings on a `space of species' appear in connection with the quantification of ecological diversity. There, it is specifically positive weightings that are of interest: they represent, after normalization, the relative abundance of species in a maximally diverse ecological community  \cite{LeinsterMeckes2016, LeinsterRoff}. The weight assigned to each species by a weighting reflects its `specialness', or the extent to which its features differentiate it from other species in the same community.

Thus, a family of weightings \((\vec{w}(t))_{t > 0}\) on \(X\) is not just a means to compute the magnitude function, but a valuable refinement of it, capturing qualitative information about the distribution of points within \(X\). However, where a single weighting is called for, it cannot be chosen canonically. First, weightings, like magnitude, are scale-dependent. Second, even having fixed a choice of \(t\), there is no guarantee that there will exist a unique weighting on \(X\) at scale \(t\). In order to make meaningful comparisons between spaces, one would prefer to assign to each finite metric space a weighting-like measure which is uniquely defined and scale-independent.

One way to find such a measure is to take a limit of the weightings \(\vec{w}(t)\), either as \(t \to \infty\) or as \(t \to 0\). To make this precise, note that for every finite metric space \(X\) the function \(t \mapsto \det(Z(t))\) is analytic, so it is nowhere locally constant. In particular, there must exist some \(\epsilon > 0\) such that \(\det(Z(t)) \neq 0\) for \(0 < t < \epsilon\). 
It follows that for \(0 < t < \epsilon\) the space \(X\) carries a unique weighting at scale \(t\), namely
\[
\vec{w}(t) = Z(t)^{-1} \onevector,
\]
where \(\onevector\) denotes the vector \(\begin{pmatrix} 1 & 1 & \cdots & 1 \end{pmatrix}^\transpose\). Moreover, the weights \(w_i(t)\) depend continuously on \(t\) in this range.  Similarly, at sufficiently large scales a unique weighting exists and depends continuously on \(t\). In fact, since \(\lim_{t \to \infty} Z(t)\) is always the identity matrix \(I\), we have
\(
\lim_{t \to \infty} \vec{w}(t) = \lim_{t \to \infty} Z(t)^{-1} \onevector = I^{-1} \onevector = \onevector
\)
for every finite metric space. 

Taking the small-scale limit is more delicate, and turns out to be much more informative. As \(t\) goes to zero, \(Z(t)\) converges to the \(n \times n\) matrix \(\onematrix\) whose entries are all 1. For \(n>1\) this matrix is singular, so there is no reason to suppose that the family of weightings \((\vec{w}(t))_{t > 0}\) will converge as \(t\) approaches zero. If it does, we will call the limiting vector a \emph{microscopic weighting}.

\begin{definition}
    Let \(X\) be a finite metric space. If the family of weightings \((\vec{w}(t))_{t > 0}\) converges component-wise as \(t \to 0\), the vector \(\lim_{t \to 0} \vec{w}(t)\) is the \define{microscopic weighting} on \(X\). We will denote it by \(\micro\).
\end{definition}

When a space admits a microscopic weighting, it is plainly unique and scale-independent. The goal, then, is to determine conditions under which a microscopic weighting exists, and what sort of information it captures.


\subsection{Outline of the paper}

We begin, in \Cref{sec:gaugings-concentration}, with some preliminary facts of linear algebra. To every symmetric real matrix one can associate a set of vectors we call \demph{gaugings} and an element of \(\R \cup \{\infty\}\) that we call the \demph{concentration} of the matrix (\Cref{def:gauging}). Weightings and magnitude can also be defined for symmetric real matrices, and they come `twinned' with gaugings and concentration in a sense explained in \Cref{rem:malue_RP1}.  Throughout our story, a central role will be played by gaugings and concentration for the \demph{distance matrix} of a finite metric space: the symmetric real matrix \(D\) with entries \(D_{ij} = d(x_i,x_j)\).

The study of microscopic weightings begins in \Cref{sec:microscopic_weightings}, where we consider the question of their existence. Our first theorem implies, in particular, that not every finite metric space admits a microscopic weighting.

\begin{thmmicrowtg_consequences}
    For every finite metric space \(X\) that admits a microscopic weighting, the following hold.
    \begin{enumerate}
        \item
        \label{intro:microwtg_consequences_1}
        The magnitude function of \(X\) satisfies \(\lim_{t \to 0} \mu_X(t) = 1\).
        
        \item
        \label{intro:microwtg_consequences_2}
        The magnitude function of \(X\) is differentiable at zero.

        \item
        \label{intro:thmmicrowtg_consequences_3}
        The microscopic weighting \(\micro\) satisfies
        \(\onevector^\transpose \micro = 1\) and \(D \micro = \magderiv{X}{0} \onevector\), where \(D\) is the distance matrix of \(X\) and \(\mu_X'\) is the first derivative of the magnitude function.
    \end{enumerate}
\end{thmmicrowtg_consequences}

In the terminology of this paper, part \eqref{intro:thmmicrowtg_consequences_3} of the theorem says that a microscopic weighting on \(X\) is a gauging for \(D\) with concentration \(\mu_X'(0)\). Part \eqref{intro:microwtg_consequences_1} says that if a space admits a microscopic weighting then it has the \emph{one-point property}. It is known that the one-point property holds generically for finite metric spaces, but not universally~\cite{RoffYoshinaga}---hence, not every space admits a microscopic weighting.

Our most general positive result concerning existence is \Cref{thm:general-existence}. To state it, let us first recall that a symmetric \(n \times n\) real matrix \(D\) determines an orthogonal splitting of \(\R^n\) as \(\im(D) \oplus \ker(D)\). We will denote by \(\R^n \xrightarrow{\proj_0} \ker(D) \xrightarrow{\inc_0} \R^n\) the projection and inclusion maps.

\begin{thmgeneral-existence}
    Let \(X\) be a finite metric space whose distance matrix \(D\) has non-zero concentration: \(\con{D} \neq 0\). Let \(D \odot D\) denote the Hadamard, or entrywise, square of \(D\). Suppose the map \(\proj_0 (D \odot D) \inc_0 \colon \ker(D) \to \ker(D)\) is invertible. Then \(X\) admits a microscopic weighting if and only if \(\con{D} \neq \infty\). If so, then
    \[
	\lim_{t \to 0} \vec{w}(t) = \con{D} (\vec{u} - \vec{u}')
	\]
	what \(\vec{u} =  D^\pseudo \onevector\) is the Moore--Penrose weighting for \(D\) (\Cref{def:moore-penrose-gauging}) and 
	\[\vec{u}' = \inc_0 \left( \proj_0 (D \odot D) \inc_0 \right)^{-1} \proj_0 \left(D \odot D\right) \vec{u}.\]
\end{thmgeneral-existence}

The conditions of the theorem hold, in particular, if \(\ker(D) = \{\zerovector\}\). Thus, we obtain \Cref{cor:special-existence}: if \(X\) is a finite space whose distance matrix is invertible, then \(X\) admits a microscopic weighting \(\micro\) if and only if \(\con{D}\) is finite. Since invertibility and finite concentration both hold generically, \Cref{cor:special-existence} implies that a generic finite metric space admits a microscopic weighting. 

It also provides a large class of interesting spaces on which a microscopic weighting is \emph{guaranteed} to exist. In \Cref{sec:negative_type} we specialize to metric spaces of negative and strictly negative type, whose definitions we recall in \Cref{sec:negative_type_spaces}. We observe in \Cref{prop:strict-neg-type-iff-mag-det} that finite spaces of strictly negative type can be characterized among those of negative type by the property that their distance matrix is invertible and has finite concentration. Thus, one can deduce from \Cref{cor:special-existence} that every finite metric space of strictly negative type admits a microscopic weighting.

\begin{thmstrict_negative_type_microwtg}
Every finite metric space \(X\) of strictly negative type admits a microscopic weighting and therefore has the one-point property. The microscopic weighting is given by
\[\lim_{t \to 0} \vec{w}(t) = \con{D} D^{-1} \onevector\]
where \(D\) is the distance matrix of \(X\). In particular, these statements hold for every finite subset of Euclidean space, every finite subset of hyperbolic space, and every finite tree.
\end{thmstrict_negative_type_microwtg}

In \Cref{sec:geometric_interpretations} we turn to the question of how to interpret the microscopic weighting when it exists. Our first interpretation rests on a theorem of Schoenberg's, which associates to every \(n\)-point metric space \(X\) of negative type a certain polytope \(\Delta(X)\): the convex hull of \(n\) points in Euclidean space. If \(X\) is of strictly negative type then \(\Delta(X)\) is an \((n-1)\)-simplex, so there exists a unique \((n-2)\)-sphere containing all its vertices. Otherwise, \(\Delta(X)\) will have dimension strictly less than \(n-1\), and in that case its vertices may or may not lie on a sphere. (See \Cref{fig:schoenberg_polytope} for an illustration.) 

We show that if \(X\) admits a microscopic weighting then the vertices of \(\Delta(X)\) do lie on a sphere. In the terminology of the distance geometry literature, this says that its distance matrix is a \demph{spherical distance matrix}. What's more, the microscopic weighting specifies the centre of the sphere in barycentric coordinates, and the radius is determined by the derivative of the magnitude function at zero.

\begin{thmneg_type_schoenberg_sphere}
Let \(X\) be a finite metric space of negative type, and suppose \(X\) admits a microscopic weighting, \(\micro\). Then, given any embedding of \(\Delta(X)\) into \(\mathbb{R}^{n-1}\), the vertices \(\vec{y}_1, \ldots, \vec{y}_n\) lie on the sphere of radius \(\sqrt{\frac{1}{2} \mu_X'(0)}\) centered at the point \(\vec{p} = \sum_{i=1}^n \widehat{w}_i(X) \vec{y}_i\).
\end{thmneg_type_schoenberg_sphere}

\Cref{thm:neg_type_schoenberg_sphere} offers an account of the microscopic weighting that is satisfying from the perspective of classical distance geometry---but one might hope for a more intrinsic interpretation, in terms of the geometry of the space \(X\) itself. For this, we follow Bj\"orck \cite{Bjoerck:Distributions-of-positive-mass}, Alexander and Solarsky \cite{Alexander-Stolarsky:Extremal-problems}, and Nickolas and Wolf \cite{Nickolas-Wolf:distance-geometry-I, Nickolas-Wolf:distance-geometry-II, Nickolas-Wolf:distance-geometry-III} in thinking of the distance function as a type of potential function.  We show in \Cref{thm:microwtg_energy_maximizing} that if a finite space \(X\) of negative type admits a microscopic weighting \(\micro\), then \(\micro\) is a canonical choice of maximizing measure for the energy integral \(I\) determined by that potential (\Cref{def:d-energy-infinite}). Moreover, the maximal energy \(M(X)\) is precisely the derivative of the magnitude function at zero.

\begin{thmmicrowtg_energy_maximizing}
Let \(X\) be a finite metric space of negative type, and suppose \(X\) admits a microscopic weighting, \(\micro\). Then \(\micro\) maximizes the energy integral \(I\) on \(X\), and
\[M(X) = \mu'_X(0).\]
In particular, this holds for every finite metric space of strictly negative type.
\end{thmmicrowtg_energy_maximizing}

In \Cref{subsec:curvature} we offer one more interpretation of the microscopic weighting, which applies only in the particular case of a finite graph. In that setting there are various candidate notions of \demph{discrete curvature}, one of which can be related to our story. We observe in \Cref{thm:microwtg-curvature} that, if one equips a finite graph \(G\) with the \demph{resistance distance}---which is always of strictly negative type---then the microscopic weighting on \(G\) records its \demph{resistance curvature}, introduced in recent work by Devriendt and Lambiotte \cite{DevriendtLambiotte}.

In the evolution of magnitude, advances in the setting of finite spaces have tended to be followed by extensions to compact spaces of negative type, which often turn out to be especially interesting for compact subsets of Euclidean space. \Cref{sec:infinite-spaces} lays the groundwork for the theory in this paper to be extended to that setting. Both the magnitude function and the distance-energy integral have been defined for compact metric spaces of negative type, so one can ask whether the relationship between them, established for finite spaces in \Cref{thm:microwtg_energy_maximizing}, holds in that generality. In \Cref{subsec:max-energy-derivative-infinite} we observe that it does hold for odd-dimensional balls in Euclidean space, and we conjecture that it holds in general for compact spaces of strictly negative type. 

To approach that conjecture by methods similar to those in Sections \ref{sec:gaugings-concentration}--\ref{sec:geometric_interpretations}, one would need to extend the definitions of microscopic weightings, gaugings and concentration so that they satisfy an analogue of \Cref{thm:microwtg_consequences}. Because the notion of weighting is more subtle for compact spaces, it is not straightforward to define microscopic weightings in general for compact spaces of strictly negative type. For subsets of Euclidean space, however, we can attempt it, using Meckes' notion of weightings as elements of a space of distributions; this is the topic of \Cref{sec:microwtg-infinite}. We close by considering the particular example of the three-dimensional ball, for which we confirm that an analogue of \Cref{thm:microwtg_consequences} does hold.

\subsection*{Acknowledgements}

We are grateful to Yasuhiko Asao and Kouichi Taira for sharing with us an unpublished note in which they consider the question of whether the existence of \(\lim_{t \to 0} \vec{w}(t)\) is equivalent to the one-point property. We are grateful, too, to Kenta Kitamura for identifying, with ChatGPT, an example which demonstrates that the two are \emph{not} equivalent (see \Cref{rem:kitamura}) and to Moritz Firsching for bringing Kitamura's contribution to our attention. Thanks are due, also, to Mark Meckes and Tom Leinster, who offered helpful comments on the first version of the paper, and to the maintainers of SageMath, which helped with much experimental work.

\subsection*{AI declaration}

The first preprint version of this paper was AI-free. In this second version, the main change is the inclusion of a stronger existence result, \Cref{thm:general-existence}, whose proof contains key elements obtained through exchanges with Claude Opus 4.8. Specifically, the LLM suggested the change of basis used in the proof and the application of the Sherman--Morrison formula. The details of the proof in its present form were worked out by the authors, who take full responsibility for its final presentation. No LLM was used in the process of writing the paper itself.


\section{Weightings and gaugings, magnitude and concentration}
\label{sec:gaugings-concentration}

We will start by defining \emph{gaugings} and \emph{concentration} for symmetric real matrices. These come `twinned' with Leinster's \emph{weightings} and \emph{magnitude} for matrices, whose definition we recall below. Indeed, as we explain in \Cref{rem:malue_RP1}, magnitude and concentration can be considered as two charts of a single \(\R P^1\)-valued invariant of symmetric real matrices. In this paper, we will be interested primarily in the concentration of distance matrices. \Cref{eg:five-points} introduces a family of metric spaces whose distance matrices we will refer to throughout the paper.


\subsection{The magnitude and the concentration of a matrix}
\label{sec:concentration_of_matrices}

First, following Leinster~\cite{LeinsterMagnitude2013}, we define weightings and magnitude for matrices. 

\begin{definition}
\label{def:matrix_weighting}
    Let \(A\) be a symmetric \(n \times n\) real matrix. A \define{weighting} for \(A\) is a vector \(\vec{w} \in \R^n\) such that \(A \vec{w} = \onevector\).
\end{definition}

If \(A\) happens to possess two weightings, \(\vec{w}\) and \(\vec{w}'\), then
\[\sum_{i=1}^n w_i' = \onevector^\transpose \vec{w}' = (A \vec{w})^\transpose \vec{w} = \vec{w}^\transpose A \vec{w}' = \vec{w}^\transpose \onevector = \sum_{i=1}^n w_i.\]
Thus, the sum of entries in a weighting is a value independent of the weighting chosen to compute it, and the following definition makes sense.

\begin{definition}
\label{def:matrix_magnitude}
    If a symmetric \(n \times n\) real matrix \(A\) has a weighting, then the \define{magnitude} of \(A\) is defined to be the real number \(\matmag{A} = \onevector^\transpose \vec{w}\) for any weighting \(\vec{w}\). If \(A\) has no weighting, \(\matmag{A}\) is defined to be \(\infty\).
\end{definition}

In the world of metric spaces, one is usually interested in the magnitude of the similarity matrix: the magnitude function of a finite metric space \(X\), as defined in \Cref{sec:introduction}, is given for \(t > 0\) by \(\mu_X(t) = \matmag{Z(t)}\), where \(Z(t)\) is the similarity matrix of \(X\) at scale \(t\). In this paper, though, we will also be interested in weightings and magnitude for the distance matrix of \(X\), and even more interested in their \demph{gaugings} and \demph{concentration}, defined as follows.

\begin{definition}
\label{def:gauging}
    Let \(A\) be a symmetric \(n \times n\) real matrix. A \define{gauging} for \(A\) is a vector \(\vec{v} \in \R^n\) such that \(\onevector^\transpose \vec{v} = 1\) and \(A \vec{v} = c \onevector\) for some \(c \in \R\).  The number \(c\) is called the \define{concentration} of the gauging \(\vec{v}\). 
\end{definition}

If \(A\) happens to possess two gaugings \(\vec{v}\) and \(\vec{v}'\), with concentrations \(c\) and \(c'\) respectively, then
\[c = c \onevector^\transpose \vec{v}' = (\vec{v}^\transpose A^\transpose) \vec{v}' = \vec{v}^\transpose (c' \onevector) = c'.\]
Thus, concentration, like magnitude, is really an attribute of the matrix \(A\).

\begin{definition}
\label{def:concentration}
    If a symmetric real matrix \(A\) has a gauging, then the \define{concentration} of \(A\), denoted \(\con{A}\), is defined to be the concentration of any gauging. Otherwise, \(\con{A}\) is defined to be \(\infty\).
\end{definition}

Observe that if \(\vec{v}\) is a gauging for a matrix \(A\) with concentration \(c\), then for each real \(t > 0\) the vector \(\vec{v}\) is also a gauging for \(tA\) with concentration \(tc\). Note also that if \(A\) is the distance matrix of a finite metric space \(X\), then \(tA\) is the distance matrix of the scaled space \(tX\). Thus, the set of gaugings for a distance matrix is scale-independent in the following sense.

\begin{lemma}
    Let \(X\) be a finite metric space. A vector \(\vec{v}\) is a gauging for the distance matrix of \(X\) if and only if \(\vec{v}\) is a gauging for the distance matrix of \(tX\) for every \(t > 0\). \qed
\end{lemma}

If \(\con{A} = 0\) then \(A\) has non-trivial kernel, so it is not invertible; we record this as a lemma for future reference.

\begin{lemma}
\label{lem:val0_det0}
    If a symmetric real matrix \(A\) satisfies \(\con{A} = 0\), then \(\det(A) = 0\). \qed
\end{lemma}

As we will see in \Cref{sec:geometric_interpretations}, vectors \(\vec{v}\) and scalars \(c\) satisfying \Cref{def:gauging} find applications in various corners of metric geometry. Despite that, they seem not to have acquired stable names. We have chosen the term \emph{concentration} for \(c\) to reflect the fact that the magnitude and the concentration of a matrix are reciprocal.

\begin{theorem}
\label{thm:mag_val_reciprocal}
    For every symmetric real matrix \(A\) we have
    \[
        \matmag{A} = 1/\con{A},
    \]
    where we take \(1/0 \coloneq \infty\) and \(1/\infty \coloneq 0\).
\end{theorem}

\begin{proof}
    If \(\matmag{A}\) and \(\con{A}\) are both finite and non-zero, it is clear that \(\vec{w}\) is a weighting for \(A\) if and only if \(\frac{1}{\matmag{A}} \vec{w}\) is a gauging for \(A\) with concentration \(\frac{1}{\matmag{A}}\), so the theorem holds. It follows that if either \(\matmag{A}\) or \(\con{A}\) is zero or infinite, the other must also be zero or infinite; we just need to see that in this case they are not equal. We will establish this by showing that \(\matmag{A} = \infty\) if and only if \(\con{A} = 0\).
    
    Recall that for every symmetric \(n \times n\) real matrix \(A\) there is an orthogonal splitting of \(\R^n\) as \(\im(A) \oplus \ker(A)\), which is the splitting into zero and non-zero eigenspaces.  In particular, there is a unique decomposition of the vector \(\onevector \in \R^n\) as \(\onevector = \vec{u} + \vec{v}\)    where \(\vec{u} = A\vec{w}\) for some \(\vec{w} \in \R^n\) (which is not necessarily unique) and \(A \vec{v} = \vec{0}\). Orthogonality says that \(0 = \vec{u}^\transpose \vec{v} = \vec{w}^\transpose A \vec{v}\), so 
    \(\onevector^\transpose \vec{v} = (\vec{u} + \vec{v})^\transpose \vec{v} =  \vec{v}^\transpose \vec{v} = \left| \vec{v}\right|^2\).

    Suppose that \(\matmag{A} = \infty\). This means precisely that \(A\) has no weighting, so \(\onevector \not \in \im(A)\) and hence \(\vec{v} \not = \vec{0}\) in the decomposition of \(\onevector\). It follows that in this case we have \(\onevector^\transpose \vec{v} = \left| \vec{v}\right|^2 \not = 0 \), and \(\bigl(1/ \left| \vec{v}\right|^2\bigr) \vec{v}\) is a gauging for \(A\) of concentration 0. Conversely, suppose that \(\con{A} = 0\). Then there is a gauging \(\vec{v}\) with \(A \vec{v} = 0\) and \(\onevector^\transpose \vec{v} = 1\), so \(\onevector\) is not orthogonal to the kernel of \(A\) and therefore cannot be in the image of \(A\). Thus, \(A\) has no weighting and \(\matmag{A} = \infty\).
\end{proof}

\begin{remark}
\label{rem:malue_RP1}
\Cref{thm:mag_val_reciprocal} entitles us to think of magnitude and concentration as two charts of a single \(\R P^1\)-valued invariant of symmetric real matrices. Given such a matrix \(A\), let us say (for the purposes of this remark only) that \(\vec{g} \in \R^n\) is a \define{gating} if 
    \[
    A \vec{g} = r \onevector
    \quad
    \text{and}
    \quad
    \onevector^\transpose \vec{g} = s
    \quad
    \text{for some } 
    (r, s) \in \R^2 \backslash (0, 0).
    \]
By~\Cref{thm:mag_val_reciprocal}, we know that every symmetric real matrix \(A\) has either a weighting or a gauging, or both, and hence has a gating. Define the \define{hinge} of \(A\) to be the element \([r : s]\) of \(\R P^1\). This is a well-defined attribute of the matrix \(A\), as if \(\vec{g}'\) is another gating then 
\[r's = \vec{g}'^\transpose \onevector s =  \vec{g}'^\transpose A \vec{g} =  s' \onevector^\transpose \vec{g} = s' r,\] 
so \([r : s] = [r' : s']\) in \(\R P^1\).  If \(s \not = 0\) then \(\frac{1}{s} \vec{g}\) is a weighting for \(A\) and \(\matmag{A} = r/s\); if \(r \not = 0\) then \(\frac{1}{r} \vec{g}\) is a gauging for \(A\) and \(\con{A} = s/r\).
\end{remark}

When a matrix admits a weighting, it may admit infinitely many. One special choice of weighting is picked out by the \emph{Moore--Penrose pseudoinverse}, which we now define. Given a symmetric \(n \times n\) real matrix \(A\), let
\[\R^n \xrightarrow{\proj_1} \im(A) \xrightarrow{\inc_1} \R^n \quad \text{and} \quad \R^n \xrightarrow{\proj_0} \ker(A)\xrightarrow{\inc_0} \R^n \]
denote the projection and inclusion maps. Given \(M \colon \R^n \to \R^n\), we will write
\[M|_{\im{A}} \coloneq \proj_1 M \inc_1 \colon \im(A) \to \im(A) \quad \text{and} \quad M|_{\ker(A)} \coloneq \proj_0 M \inc_0 \colon \ker(A) \to \ker(A).\]

\begin{definition}\label{def:moore-penrose-pseudo}
	The \define{Moore--Penrose pseudoinverse} \(A^\pseudo\) of \(A\) is the map 
	\[\inc_1 (A|_{\im(A)})^{-1} \proj_1 \colon \R^n \to \R^n.\]
\end{definition}

The next lemma can be found in Akkaya and \"Unl\"u \cite[Theorem~2.6]{AkkayaUnlu} and in Chen and Vigneux \cite[Corollary~4.8]{ChenVigneaux}.

\begin{lemma}
\label{lem:pseudo-magnitude}
	Let \(A \in \R^{n\times n}\) be symmetric and suppose \(\matmag{A} \neq \infty\). Then \(A^\pseudo\onevector\) is a weighting for \(A\) and we have \(\matmag{A} = \vec{v}^\transpose (A|_{\im(A)})^{-1} \vec{v}\) where \(\vec{v} = \proj_1 \onevector\).
\end{lemma}

\begin{proof}
	The condition that \(\matmag{A} \neq \infty\) says precisely that \(\onevector \in \im(A)\). The composite \(A A^\pseudo\) is the projection to \(\im(A) \subseteq \R^n\), so we have \(A A^\pseudo \onevector = \onevector\), which says that \(A^\pseudo \onevector\) is a weighting for \(A\). From this, we see that
	\[\vec{v}^\transpose (A|_{\im(A)})^{-1} \vec{v} = \onevector^\transpose (\proj_1)^\transpose (A|_{\im(A)})^{-1} \proj_1 \onevector = \onevector^\transpose \inc_1 (A|_{\im(A)})^{-1} \proj_1 \onevector = \onevector^\transpose (A^\pseudo \onevector) = \matmag{A}. \qedhere\]
\end{proof}

It will be useful to give this particular weighting a name.

\begin{definition}\label{def:moore-penrose-gauging}
	Let \(A\) be a symmetric real matrix such that \(\matmag{A} \neq \infty\). The \define{Moore--Penrose weighting} for \(A\) is the vector \(A^\pseudo \onevector\). The \define{Moore--Penrose gauging} for \(A\) is its normalization, \(\con{A} A^\pseudo \onevector\).
\end{definition}

We close this section by introducing a two-parameter family of five-point metric spaces that will provide a valuable source of examples throughout the paper.

\tikzstyle{point} = [circle, draw, fill=blue!80, 
    text width=0.5em, text centered, 
    node distance=3cm, inner sep=0pt]
\tikzstyle{line} = [draw, -]

\begin{figure}
    \centering
    \begin{tikzpicture}[baseline=.1em]
	\coordinate (A1) at (.3,1.7); 
	\coordinate (A2) at (2.7,1.7); 
	\coordinate (B1) at (-.5,0); 
	\coordinate (B2) at (1.5,-1.5); 
	\coordinate (B3) at (3.5,0); 

	\draw [thick] (A1)--(B1)--(A2)--(B2)--(A1)--(B3)--(A2);
	\draw[thick] (A1) -- node[above] {$b$} (A2);
	\draw[thick] (B1) -- node[above] {$a$} (B2);
	\draw[thick] (B2) -- node[above] {$a$} (B3);
	\draw[thick] (B3) -- node[above] {$a$} (B1);
	\filldraw[draw=black, fill=black] (A1) circle [radius=0.1]; 
	\filldraw[draw=black, fill=black] (A2) circle [radius=0.1]; 
	\filldraw[draw=black, fill=black] (B1) circle [radius=0.1]; 
	\filldraw[draw=black, fill=black] (B2) circle [radius=0.1]; 
	\filldraw[draw=black, fill=black] (B3) circle [radius=0.1]; 
    \end{tikzpicture}
    \qquad
    \begin{tikzpicture}[baseline=5em]
        \begin{axis}[
            width = 0.5\textwidth,
            axis equal image=true,
            axis x line=center, 
            axis y line=left,
            axis line shift=10pt, 
            xmin=0.95, ymin=0.95, 
            ymax=2.2, xmax=2.2,
            xtick={0,...,2}, 
            xticklabels={0, 1, 2},
            ytick={1,2}, 
            yticklabels={1,2},
            x axis line style={style = -}, 
            y axis line style={style = -},
            xlabel=$a$, 
            ylabel=$b$,
            y label style={at={(current axis.above origin)}, anchor=south, rotate=-90},
            ]
            \fill[fill=blue!20] (0.95, 0.95) -- (0.95, 2) -- (1.5, 2) -- (2, 4/3) -- (2, 0.95) -- cycle;
            \fill[fill=white] (2, 2) --  (1.5, 2) -- (2, 4/3) --  cycle;
            \addplot [color=red, ultra thick, dotted, domain=1.5:2] {4 - 4/3*x}; 
            \addplot [color=blue, ultra thick, dashed, domain=1.5:2] {3/x};
            \addplot [domain=0:2] {2};
            \addplot [color=black, thin, domain=0:2] ({2}, {x});
            \node [point, label=right:{\(P_{\mathrm{C}}\)}] (emily) at (2, 4/3) {};
            \node [point, label={\(P_{\mathrm{B}}\)}] (K) at (2, 2) {};
            \node [point, label={\(P_{\mathrm{E}}\)}] (W) at (3/2, 2) {};
            \node [point, label=right:{\(P_{\mathrm{D}}\)}] (detzero) at (2, 3/2) {};
            \node [point, label=right:{\(P_{\mathrm{A}}\)}] (euclidean) at ({sqrt(3/2)}, 1) {};
        \end{axis}
    \end{tikzpicture}
    \caption{The five-point metric space \(P(a,b)\) described in \Cref{eg:five-points}. Unmarked edges represent distances equal to 1, while \(a\) and \(b\) can take any values in the interval \((0,2]\). On the right is part of the parameter space. See Examples \ref{eg:five-points}, \ref{eg:5pt-spaces-1pt-prop} and \ref{eg:five-point-microwtgs} for discussion.
    }
    \label{fig:graph-family}
\end{figure}
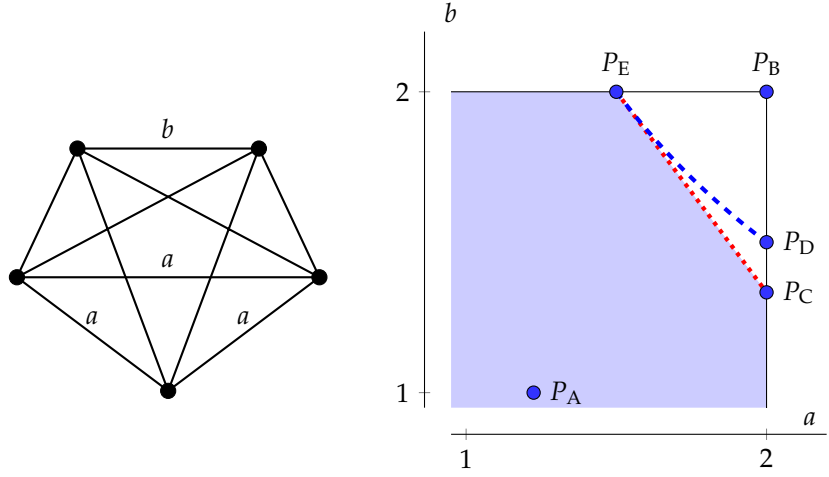

\begin{example}
\label{eg:five-points}
    For \(0 < a, b \le 2\), let \(P(a, b)\) be the metric space with distance matrix
\[
    D^{(a, b)}
    =
    \begin{pmatrix}
        0 & a & a & 1 & 1 \\
        a & 0 & a & 1 & 1 \\
        a & a & 0 & 1 & 1 \\
        1 & 1 & 1 & 0 & b \\
        1 & 1 & 1 & b & 0 \\
    \end{pmatrix}.
\]
    We can picture \(P(a, b)\) in terms of the graph in \Cref{fig:graph-family}; the unmarked edges have length \(1\).  This figure also shows the part of the parameter space where \(1 \le a, b \le 2\).

    For \(0 < a,b \le 2\), we have \(\det(D^{(a, b)}) = 2(3 - ab)a^2b\), while
    \[
    \con*{D^{(a, b)}} = 
        \begin{cases}
            1
            &
            \text{if } (a, b) = (3/2, 2)
            \\[0.5em]
            \dfrac{2(ab-3)}{4a + 3b - 12} 
            & 
            \text{otherwise.}
            \\[0.4em]
        \end{cases}
    \]    
    A gauging for \(D^{(a, b)}\) is given by 
    \[
        \vec{v} 
        = 
        \begin{cases}
            \tfrac{1}{4a + 3b - 12}(b-2, b-2, b-2, 2a - 3, 2a - 3)^\transpose
            &
            \text{when } 4a + 3b \neq 12
            \\[0.1em]
            (1/5, 1/5, 1/5, 1/5, 1/5)^\transpose
            &
            \text{when } (a, b) = (3/2, 2);    
        \end{cases}
    \]
    when \(4a+3b = 12\) and \((a,b) \neq (3/2,2)\), there is no gauging. When \((a,b) = (3/2,2)\) the gauging is not unique, because \(D^{(a,b)}\) is singular; we have given the Moore--Penrose gauging here, and we will see a different canonical gauging in \Cref{eg:five-point-microwtgs}.
    
    In \Cref{fig:graph-family} the dotted red line is the line where \(4a + 3b = 12\); here, every space but one has \(\con{D^{(a,b)}} = \infty\). The dashed blue curve is the curve where \(ab = 3\); here, every space but one has \(\con{D^{(a,b)}} = 0\). These also happen to be the only spaces in this family for which \(\det(D^{(a,b)}) = 0\). The exceptional space is \(P_{\mathrm{E}} = P(3/2, 2)\), which lies at the intersection of the red and blue lines, but has \(\con{D^{(3/2, 2)}} = 1\).

    We will return repeatedly to this family of spaces, looking more closely at \(P_{\mathrm{A}}\), \(P_{\mathrm{B}}\), \(P_{\mathrm{C}}\), \(P_{\mathrm{D}}\) and \(P_{\mathrm{E}}\) in \Cref{eg:5pt-spaces-1pt-prop}. In \Cref{prop:PQ-negative-type} we will see that the shaded area is where the metric spaces are of strictly negative type, while spaces on the dotted red line are of negative type, but not of strictly negative type.
\end{example}


\section{Microscopic weightings}
\label{sec:microscopic_weightings}

We turn now to the main character of our story: the \emph{microscopic weighting} on a finite metric space. Recall from \Cref{sec:introduction} that for all sufficiently small \(t > 0\) the similarity matrix \(Z(t)\) is invertible, so the scaled space \(tX\) carries the unique weighting \(\vec{w}(t) = Z(t)^{-1} \onevector\). The microscopic weighting on \(X\) is defined to be
\[\micro \coloneq \lim_{t \to 0} \vec{w}(t)\]
when that limit exists. In this section, we consider when a microscopic weighting is guaranteed to exist.

In \Cref{subsec:necessary-conditions} we record some immediate observations which imply, in particular, that not quite every finite metric space admits a microscopic weighting---and which begin to account for our interest in gaugings and concentration. \Cref{subsec:sufficient-conditions} contains our strongest positive result concerning existence: \Cref{thm:general-existence} says that if \(X\) is a space whose distance matrix \(D\) satisfies \(\con{D} \neq 0\), and the Hadamard product \(D \odot D\) is invertible when restricted and corestricted to \(\ker(D)\), then \(X\) carries a microscopic weighting if and only if \(\con{D} \neq \infty\). In that case, we can give a formula to compute the microscopic weighting directly from the distance matrix.

The conditions of \Cref{thm:general-existence} hold, in particular, in the case that \(D\) is invertible. As we explain in \Cref{rem:almost-every}, this implies that a microscopic weighting exists generically. In \Cref{sec:negative_type} we will see that it also guarantees the existence of a microscopic weighting on every finite metric space of strictly negative type.


\subsection{Microscopic weightings and the one-point property}
\label{subsec:necessary-conditions}

A metric space \(X\) is said to have the \demph{one-point property} if its magnitude function satisfies \(\lim_{t \to 0} \mu_X(t) = 1\). Given the interpretation of \(\mu_X(t)\) as recording the effective number of points in \(X\) at scale \(t\), you might expect that every finite space should have the one-point property: that at very small scales every finite space should be `effectively' a single point. This is very nearly the case, but not quite. 

That the one-point property can fail was demonstrated first by an example due to Willerton \cite[Example~2.2.8]{LeinsterMagnitude2013}. Roff and Yoshinaga have since shown that almost every finite metric space has the one-point property \cite[Theorem~2.3]{RoffYoshinaga}, but that its failure can be arbitrarily bad, in the sense that \(\lim_{t \to 0} \mu_X(t)\) can take any value in the interval \([1, \infty)\) \cite[Theorem~3.8]{RoffYoshinaga}. More recently still, Kamiyama \cite{Kamiyama2026} has demonstrated that \(\lim_{t \to 0} \mu_X(t)\) can also be strictly smaller than 1. These pathologies have wider consequences for the continuity of magnitude, which are explored by Katsumasa, Roff and Yoshinaga in \cite{KatsumasaRoffYoshinaga}.

Our first result concerning existence says, among other things, that if a space admits a microscopic weighting, then it has the one-point property. Thus, not every finite metric space does admit a microscopic weighting. 

\begin{theorem}
\label{thm:microwtg_consequences}
For every finite metric space \(X\) that admits a microscopic weighting, the following hold.
    \begin{enumerate}
        \item
        \label{thm:microwtg_consequences_1}
        The space \(X\) has the one-point property: \(\lim_{t \to 0} \mu_X(t) = 1\).
        
        \item
        \label{thm:microwtg_consequences_2}
        The magnitude function \(\mu_X\) is differentiable at zero.

        \item
        \label{thm:microwtg_consequences_3}
        The microscopic weighting is a gauging for \(D\) and \(\con{D} = \mu_X'(0)\), where \(\mu_X'\) is the first derivative of \(\mu_X\).
    \end{enumerate}
\end{theorem}

\begin{proof}
    We know that \(Z(t)\) is invertible for all sufficiently small \(t>0\), so we assume \(t\) is in that range here.  We also know that \(\lim_{t \to 0} Z(t) = \onevector \onevector^\transpose\), the matrix whose entries are all 1. Thus, if \(X\) admits the microscopic weighting \(\micro\), we have
    \[
        \lim_{t \to 0} \bigl(Z(t) \vec{w}(t)\bigr) 
        =
        \lim_{t \to 0} Z(t) \lim_{t \to 0}\vec{w}(t) 
        =
        (\onevector \onevector^\transpose) \micro
        = 
        \onevector (\onevector^\transpose \micro)
        =
        (\onevector^\transpose \micro) \onevector.
    \]
    Taking the limit as \(t \to 0\) of the parametrized weight equation
    \begin{equation}
    \label{eq:parametrized_weight_equation}
        Z(t)\vec{w}(t) = \onevector,
    \end{equation}
    we see that \(\onevector^\transpose \micro = 1\), so the sum of entries in \(\micro\) is \(1\). Thus,
    \[
        \lim_{t \to 0}\mu_X(t)
        =
        \lim_{t \to 0} \bigl( \onevector^\transpose \vec{w}(t) \bigr)
        =
        \onevector^\transpose \Bigl( \lim_{t \to 0} \vec{w}(t) \Bigr)
        =
        \onevector^\transpose \micro
        =
        1,
    \]
    which says that \(X\) has the one-point property. This proves part (\ref{thm:microwtg_consequences_1}).

    For parts~(\ref{thm:microwtg_consequences_2}) and~(\ref{thm:microwtg_consequences_3}), differentiate equation \eqref{eq:parametrized_weight_equation} entrywise to see that
    \begin{equation}
    \label{eq:differentiate_1}
        -Z'(t)\vec{w}(t) = Z(t)\vec{w}'(t).
    \end{equation}
    We can write \(Z(t)\) as \(\exp_{{\odot}}(-tD)\), where  \(\odot\) denotes the Hadamard, or entrywise, product, so differentiating this gives
    \(
        Z'(t) = -D \odot Z(t)
    \).
    Hence,
    \(Z'(0) = -D\), and the left-hand side of~\eqref{eq:differentiate_1} converges  to \(D\micro\) as \(t \to 0\).  Taking the limit as \(t \to 0\) in~\eqref{eq:differentiate_1} we obtain
    \begin{align*}
        D\micro 
        &= \lim_{t\to 0} Z(t) \lim_{t\to 0}\vec{w}'(t) 
        = \onevector \onevector^\transpose \lim_{t \to 0} \vec{w}'(t)
        = \onevector \Bigl(\lim_{t \to 0} \onevector^\transpose \vec{w}'(t)\Bigr)
        \\
        &=
        \Bigl(\lim_{t \to 0} \onevector^\transpose \tfrac{\operatorname{d}}{\operatorname{d}t}\vec{w}(t) \Bigr) \onevector
        =
        \Bigl(\lim_{t \to 0} \tfrac{\operatorname{d}}{\operatorname{d}t}\bigl(\onevector^\transpose \vec{w}(t)\bigr) \Bigr) \onevector
        =
        \Bigl(\lim_{t \to 0} \mu_X'(t) \Bigr) \onevector.
    \end{align*}
    Since \(\onevector^\transpose \micro = 1\) by (\ref{thm:microwtg_consequences_1}), it follows that  \(\micro\) is a gauging for \(D\), and
    \(\con{D} = \lim_{t \to 0} \mu_X'(t) = \mu_X'(0)\).
\end{proof}

From Part (\ref{thm:microwtg_consequences_3}) of \Cref{thm:microwtg_consequences}, we obtain the following statement---since, if \(\con{D} = \infty\), then \(D\) has no gauging.

\begin{prop}
\label{prop:magD_zero_microweighting_does_not_exist}
    Let \(X\) be a finite metric space whose distance matrix has infinite concentration. Then \(X\) does not admit a microscopic weighting. \qed
\end{prop}

\begin{remark}\label{rem:kitamura}
    In the first preprint version of this paper, we conjectured that having infinite concentration might be the only obstruction to the existence of a microscopic weighting. That has turned out not to be the case. Assisted by ChatGPT 5.6 sol, Kenta Kitamura has constructed a ten-point space which has finite, non-zero concentration but no microscopic weighting~\cite{BlairEtAl:formal-conjectures-website}. One especially interesting feature of Kitamura's example is that it \emph{does} have the one-point property. Thus, while part \eqref{thm:microwtg_consequences_1} of \Cref{thm:microwtg_consequences} says that the existence of a microscopic weighting on \(X\) implies the one-point property, Kitamura's example demonstrates that the converse is not true.
\end{remark}

\begin{figure}
\centering
    \pgfplotsset{
        every axis/.style={
            xmin=0, xmax=0.55,
            ymin=0.5, ymax=2.1,
            width = 0.6\textwidth,
            axis equal image=true,
            axis x line=bottom, 
            axis y line = left,
            xtick={0, 0.5}, 
            xticklabels={$0$, $0.5$},
            ytick={0.5, 1, 1.5, 2}, 
            yticklabels={0.5, 1, 1.5, 2},
            x axis line style={style = -},
            y axis line style={style = -},
        },
    } 
    \begin{tikzpicture}[baseline=5em]
            \begin{axis}[title={\(P_{\mathrm{A}}\)}]
                \addplot[mark=none] file {mag_limit_E.dat};
                \addplot[mark=none,red,dashed] expression {1};
            \end{axis}
    \end{tikzpicture}
    ~
    \begin{tikzpicture}[baseline=5em]
            \begin{axis}[title={\(P_{\mathrm{B}}\)}]
                \addplot[mark=none] file {mag_limit_B_1.dat};
                \addplot[mark=none] file {mag_limit_B_2.dat};
                \addplot[mark=none,red,dashed] expression {1};
            \end{axis}
    \end{tikzpicture}
    ~
    \begin{tikzpicture}[baseline=5em]
            \begin{axis}[title={\(P_{\mathrm{C}}\)}]
                \addplot[mark=none] file {mag_limit_A.dat};
                \addplot[mark=none,red,dashed] expression {1};
            \end{axis}
    \end{tikzpicture}
    ~
    \begin{tikzpicture}[baseline=5em]
            \begin{axis}[title={\(P_{\mathrm{D}}\)}]
                \addplot[mark=none] file {mag_limit_D_1.dat};
                \addplot[mark=none] file {mag_limit_D_2.dat};
                \addplot[mark=none,red,dashed] expression {1};
            \end{axis}
    \end{tikzpicture}
    \caption{The small-scale behaviour of the magnitude function for four spaces in the family \(P(a,b)\). See \Cref{eg:5pt-spaces-1pt-prop} for discussion.}
    \label{fig:magnitude-function-plots}
\end{figure}
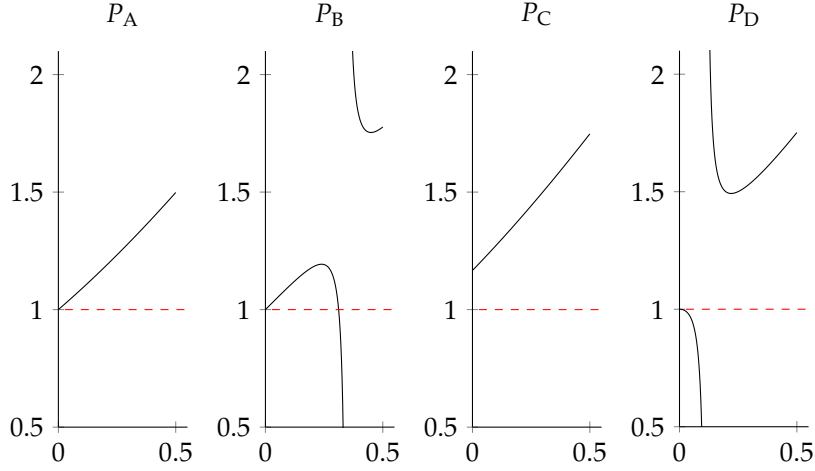

\begin{example}
\label{eg:5pt-spaces-1pt-prop} 
    In \Cref{fig:magnitude-function-plots} we plot the magnitude function for small \(t\), for the spaces corresponding to the points labelled \(P_{\mathrm{A}}\), \(P_{\mathrm{B}}\), \(P_{\mathrm{C}}\) and \(P_{\mathrm{D}}\) in \Cref{fig:graph-family}. Using the computations in \Cref{eg:five-points} we can make the following observations.

    For \(P_{\mathrm{A}} = P(\sqrt{3/2}, 1)\) and \(P_{\mathrm{B}} = P(2,2)\) the distance matrix has finite and non-zero concentration, and we can see from the plots of their magnitude functions that the one-point property holds in each case. In other respects, these spaces are quite different. The space \(P_{\mathrm{A}}\) is as nice as possible: it embeds isometrically into Euclidean space \(\R^4\), so is of strictly negative type. The space \(P_{\mathrm{B}}\) is given by the shortest-path metric on the complete bipartite graph \(K_{3,2}\), and is not even of negative type.

    The distance matrix of \(P_{\mathrm{C}} = P(2,4/3)\) has infinite concentration, and it is clear from the plot of its magnitude function that \(P_{\mathrm{C}}\) lacks the one-point property. (This space appears in \cite[Example~3.4]{RoffYoshinaga} as a `smallest' space for which the one-point property fails.) Meanwhile, the distance matrix of \(P_{\mathrm{D}}=P(2,3/2)\) has zero concentration, but the one-point property does hold for \(P_{\mathrm{D}}\). Thus, \(\con{D} = 0\) does not imply failure of the one-point property. 
\end{example}


\subsection{An existence theorem and a formula for the microscopic weighting}
\label{subsec:sufficient-conditions}

In this section, we will state our most general positive result concerning the existence of a microscopic weighting. First, we  fix some notation and record a proposition that will be key to our analysis of the small-\(t\) asymptotics of \(Z(t)\) and its inverse. 

\begin{definition}\label{def:big-o}
    A function \(f \colon (0,\infty) \to \R\) is said to be \(O(t^p)\) as \(t \to 0\) if there exists a real number \(c > 0\) such that \(|f(t)| \leq c \, t^p\) for all sufficiently small \(t > 0\). A \(t\)-dependent matrix \(A\) (that is, a function \(A \colon (0, \infty) \to \R^{n \times m}\)) will be said to be \(O(t^p)\) as \(t \to 0\) if every entry is \(O(t^p)\) as \(t \to 0\). Explicitly, that means there exists a real number \(C > 0\) such that for each \(i,j\) we have \(|A(t)_{ij}| \leq C \, t^p\) for all sufficiently small \(t > 0\). A \(t\)-dependent linear map from \(V\) to \(W\) (that is, a function \((0, \infty) \to \hom(V,W)\)) will be said to be \(O(t^p)\) if its matrix with respect to some (and hence any) choice of bases is \(O(t^p)\).
\end{definition}

The following proposition, which underlies the proof of the main theorem, is proved in \Cref{subsec:special-existence-proof}.

\begin{prop}
\label{prop:UL-inverse}
Let \(D \in \R^{n \times n}\) be invertible and let \(B(t) \in\R^{n \times n}\) be \(O(t^2)\) as \(t \to 0\). Given any vector \(\vec{v} \in \R^n\) such that \(\vec{v}^\transpose D^{-1} \vec{v} \neq 0\), the following statements hold.
\begin{enumerate}
    \item For all sufficiently small \(t > 0\), the matrix \(\vec{v}\vec{v}^\transpose - tD + B(t)\) is invertible. \label{pt:UL-inverse1}
    \item As \(t \to 0\), we have \(\left(\vec{v}\vec{v}^\transpose - tD + B(t) \right)^{-1} = O(t^{-1})\). \label{pt:UL-inverse2}
    \item As \(t \to 0\), we have \(\left(\vec{v}\vec{v}^\transpose - tD + B(t) \right)^{-1} \vec{v} = (\vec{v}^\transpose D^{-1} \vec{v})^{-1} D^{-1} \vec{v} + O(t)\). \label{pt:UL-inverse3}
\end{enumerate}
\end{prop}

To state the theorem, we first recall that a symmetric matrix \(A \in \R^{n\times n}\) determines an orthogonal splitting of \(\R^n\) as \(\im(A) \oplus \ker(A)\); the projection and inclusion maps are denoted by
\[\R^n \xrightarrow{\proj_1} \im(A) \xrightarrow{\inc_1} \R^n \quad \text{and} \quad \R^n \xrightarrow{\proj_0} \ker(A)\xrightarrow{\inc_0} \R^n.\]
Given \(M \colon \R^n \to \R^n\), we will write \(M|_{\im{A}}\) for the map \(\proj_1 M \inc_1 \colon \im(A) \to \im(A)\) and \(M|_{\ker(A)}\) for the map \(\proj_0 M \inc_0 \colon \ker(A) \to \ker(A)\). Notice that the map \(A|_{\im(A)}\) is guaranteed to be invertible, as \(\im (A) \cap \ker (A) = \{\zerovector\}\).

\begin{theorem}
\label{thm:general-existence}
    Let \(X\) be a finite metric space whose distance matrix \(D\) satisfies \(\con{D} \neq 0\), and suppose that \((D \odot D)|_{\ker(D)}\) is invertible. Then \(X\) admits a microscopic weighting if and only if \(\con{D} \neq \infty\). If so, then
    \[
	\lim_{t \to 0} \vec{w}(t) = \con{D} (\vec{u} - \vec{u}')
	\]
	where \(\vec{u} =  D^\pseudo \onevector\) is the Moore--Penrose weighting for \(D\) (\Cref{def:moore-penrose-gauging}) and 
	\[\vec{u}' = \inc_0 \left(D \odot D|_{\ker(D)} \right)^{-1} \proj_0 \left(D \odot D\right) \vec{u}.\]
\end{theorem}

The detailed proof of \Cref{thm:general-existence} can be found in \Cref{subsec:general-existence-proof}. For now, we give a sketch that highlights the role played by each of the assumptions on the distance matrix.

\begin{proof}[Sketch Proof]
The first step is to take the Taylor expansion of the matrix \(Z(t)\), entrywise, at \(t = 0\). This yields the Hadamard power series
\begin{equation}\label{eq:hadamard-series}
	\exp_{\odot}(-tD) = \onematrix - Dt + \frac{1}{2} D^{\odot 2} t^2 - \frac{1}{3!} D^{\odot 3} t^3 + \cdots
\end{equation}
and tells us in particular that 
\[\exp_{\odot}(-tD) =  \onematrix - Dt + B(t)\] 
where \(B(t) = O(t^2)\) as \(t \to 0\).

Suppose, for a moment, that \(D\) is invertible. Since \(\onevector^\transpose D^{-1} \onevector = 1/\con{D}\), the assumption that \(\con{D} \neq \infty\) guarantees that \(\onevector^\transpose D^{-1} \onevector \neq 0\), and we can apply part \eqref{pt:UL-inverse3} of \Cref{prop:UL-inverse} to see that, as \(t \to 0\), we have
    \[\vec{w}(t) = \exp_{\odot}(-tD)^{-1} \onevector = \left(\onematrix - Dt + B(t) \right)^{-1} \onevector = \con{D} D^{-1} \onevector + O(t).\]
Thus, if \(D\) is invertible, \(\lim_{t \to 0} \vec{w}(t) = \con{D} D^{-1} \onevector = \con{D} \vec{u}\), as claimed.

If \(D\) is not invertible, we cannot apply \Cref{prop:UL-inverse} to the expansion in equation \eqref{eq:hadamard-series} directly. The key in this case is to express \(\exp_\odot (-tD)\) in block matrix form with respect to the orthogonal decomposition of \(\R^{n}\) as \(\im(D) \oplus \ker(D)\). We can then apply the \textit{Schur complement inversion formula} to compute the weighting \(\vec{w}(t) = \exp_{\odot} (-tD)^{-1} \onevector\) for small \(t\), using the general fact that \(D|_{\im(D)}\) is invertible and the assumption that \((D \odot D)|_{\ker(D)}\) is invertible.

As we recall in \Cref{subsec:general-existence-proof}, the Schur complement inversion formula specifies how to invert a block matrix, provided that the upper-left block and its Schur complement are both invertible. In our case, the upper-left block has the form 
\[
    \vec{v} \vec{v}^\transpose - D|_{\im(D)} \, t + B'(t)
\]
where \(B'(t) = O(t^2)\) as \(t \to 0\). Here, \(D|_{\im(D)}\) is guaranteed to be invertible. The vector \(\vec{v}\) is the projection of \(\onevector\) to \(\im(D)\), which is non-zero because \(\con{D} \neq 0\), so \(\matmag{D} \neq \infty\) and \(D\) has a weighting. The condition that \(\con{D} \neq \infty\) is equivalent to imposing that \(\vec{v}^\transpose (D|_{\im(D)})^{-1} \vec{v} \neq 0\). Thus, we can apply \Cref{prop:UL-inverse} to see that, as \(t \to 0\), the upper-left block is eventually invertible. We show that its Schur complement has the form 
\[t^2 \left(\frac{1}{2} (D \odot D)|_{\ker(D)} + O(t) \right).\]
The assumption that \((D \odot D)|_{\ker(D)}\) is invertible ensures this is invertible for sufficiently small \(t > 0\), so we can apply the inversion formula to compute \(\vec{w}(t) = \exp_{\odot} (-tD)^{-1} \onevector\) for small \(t\). The result then follows by taking the limit as \(t \to 0\).
\end{proof}

If \(D\) is invertible, then \(\con{D} \neq 0\) by \Cref{lem:val0_det0} and the condition involving \(\ker(D)\) is satisfied trivially. Thus, \Cref{thm:general-existence} has the following immediate corollary, which we will use in \Cref{sec:negative_type_spaces} to see that every finite space of strictly negative type admits a microscopic weighting.

\begin{cor}
\label{cor:special-existence}
    Let \(X\) be a finite metric space whose distance matrix \(D\) is invertible. Then \(X\) admits a microscopic weighting if and only if \(\con{D} \neq \infty\). If so, then 
    \(\lim_{t \to 0} \vec{w}(t) = \con{D} D^{-1} \onevector\).
    That is, the microscopic weighting is the unique gauging for \(D\).
\end{cor}

\begin{remark}\label{rem:almost-every}
    \Cref{cor:special-existence} can also be proved directly, by arguments similar to those given by Roff and Yoshinaga in \cite[Theorem~2.3]{RoffYoshinaga}. The statement of that theorem is that the one-point property holds for \emph{almost every} finite metric space: the space of all \(n\)-point metric spaces contains a dense open subset on which the property holds. To prove this, the authors exhibit, for each \(n > 0\), a polynomial \(F_n\) of \(n^2\) variables, such that if the distance matrix \(D\) of a space \(X\) satisfies \(F_n(D) \neq 0\), then \(X\) has the one-point property. Concretely, \(F_n(D) = \sum_{j=1}^n (\adj(D)\onevector)_j\), and a little linear algebra shows that if \(D\) is invertible this is equal to \({\det(D)}/{\con{D}}\). Thus, the conditions that \(D\) is invertible and \(\con{D} \neq \infty\) guarantee that \(F_n(D) \neq 0\). This tells us that, in the sense of Roff and Yoshinaga's paper, almost every finite metric space admits a microscopic weighting.
\end{remark}

We are not yet able to prove an existence theorem without the assumption that the distance matrix satisfies \(\con{D} \neq 0\). However, as the following example shows, having \(\con{D} = 0\) is not an obstruction to the existence of a microscopic weighting. The example also serves to illustrate that \Cref{thm:general-existence} is indeed strictly more general than \Cref{cor:special-existence}, and that the microscopic weighting on a space need not be the Moore--Penrose gauging for its distance matrix.

\begin{example}\label{eg:five-point-microwtgs}
    We return to the family of spaces \(P(a,b)\) introduced in \Cref{eg:five-points}. Let \(\Xi_P \coloneq [1,2] \times [1, 2]\) denote the part of the parameter space depicted in \Cref{fig:graph-family}. Denote the dotted red line in the figure by \(C_1 \coloneq \{(a,b) \mid 12 = 4a + 3b\} \cap \Xi_P\) and the dashed blue curve by \(C_2 \coloneq \{(a,b) \mid ab = 3\} \cap \Xi_P\). Note that \(C_1\) and \(C_2\) intersect at the point \(P_\mathrm{E}\) representing the space \(P(3/2,2)\). 
    \begin{itemize}
        \item Every space that does not lie on the curve \(C_2\) has invertible distance matrix. In the generic part of the parameter space, \(\Xi_P \setminus (C_1 \cup C_2)\), the distance matrix satisfies \(\con{D} \neq \infty\), so \Cref{cor:special-existence} tells us that each of these spaces admits a microscopic weighting, namely the unique gauging \(\con{D}D^{-1} \onevector\) specified in \Cref{eg:five-points}.
        
        \item For the spaces lying on \(C_1 \setminus \{P_\mathrm{E}\}\) the distance matrix has \(\con{D} = \infty\), so \Cref{prop:magD_zero_microweighting_does_not_exist} says that none of these spaces admits a microscopic weighting.

        \item On \(C_2\), every space has singular distance matrix. On \(C_2 \setminus \{P_\mathrm{E}\}\) the distance matrix satisfies \(\con{D} = 0\), so \Cref{thm:general-existence} does not apply. Even so, these spaces each admit a microscopic weighting: calculating the limit in SageMath, one finds that for \((a,b) \in C_2 \setminus \{P_\mathrm{E}\}\) the microscopic weighting on \(P(a,b)\) is
            \[
                \micro_{(a,b)} 
                = 
                \tfrac{1}{3-2a}
                \bigl(
                    1, 1, 1, -a, -a 
                \bigr)^\transpose
            .\]
        For these spaces, we do not know how to compute \(\micro_{(a,b)}\) directly from the distance matrix.
    
        \item The space \(P_\mathrm{E} = P(3/2,2)\) is exceptional. Its distance matrix has \(\con{D} = 1\) and \(\det((D^{\odot 2})|_{\ker(D)}) \neq 0\), so we can apply \Cref{thm:general-existence} in this case to see that \(P(3/2,2)\) admits the microscopic weighting
            \begin{align*}
                \micro_{(3/2,2)} 
                = 
                \con{D} (D^\pseudo\onevector - \vec{u}')
                &= 
                \bigl(
                    \tfrac{1}{5}, \tfrac{1}{5}, \tfrac{1}{5}, \tfrac{1}{5}, \tfrac{1}{5}
                \bigr)^\transpose
                - 
                \bigl(
                    \tfrac{-1}{45}, \tfrac{-1}{45}, \tfrac{-1}{45}, 
                    \tfrac{1}{30}, \tfrac{1}{30}
                \bigr)^\transpose \\
               & =
                \bigl(
                    \tfrac{2}{9}, \tfrac{2}{9}, \tfrac{2}{9},
                    \tfrac{1}{6}, \tfrac{1}{6}
                \bigr)^\transpose.
            \end{align*}
        In particular, the microscopic weighting in this case is \emph{not} the Moore--Penrose gauging \(\con{D}D^\pseudo \onevector\). Whereas \(\con{D}D^\pseudo \onevector\) is uniformly distributed over the points of the space, \(\micro_{(3/2,2)}\) is more responsive to the geometry, reflecting that \(P(3/2,2)\) has two homogeneous subspaces but is not itself homogeneous.
    \end{itemize}
\end{example}

\begin{figure}[]
    \centering
    \begin{tikzpicture}[baseline=5em]
        \begin{axis}[
            width = 0.5\textwidth,
            axis equal image=true,
            axis x line=center, 
            axis y line=left,
            axis line shift=10pt, 
            xmin=0, xmax=1.1,
            ymin=0.95, ymax=2.1,
            xtick={0,...,2}, 
            xticklabels={0, 1, 2},
            ytick={1,2}, 
            yticklabels={1,2},
            x axis line style={style = -}, 
            y axis line style={style = -},
            xlabel=$a$, 
            ylabel=$b$,
            y label style={at={(current axis.above origin)}, anchor=south, rotate=-90},
            ]
            \filldraw[color=blue!5] (0, 0.95) rectangle (1, 2);
            \addplot [color=red, ultra thick, dotted, domain=0.6666:1] {3 - 3/2*x}; 
            \addplot [color=blue, ultra thick, dashed, domain=0.75:1] {3/(2*x)};
            \node [point, label={below:\(Q_\mathrm{E}\)}] (K) at (1,{3/2}) {};
        \end{axis}
    \end{tikzpicture}
    \caption{Part of the parameter space for the family \(\bigl\{Q(a,b) \mid (a, b) \in (0, 1] \times (0, 2]\bigr\}\) of seven-point metric spaces described in \Cref{eg:seven-points}. The red dotted line represents where the concentration is infinite; the blue dashed line represents where the concentration is zero.}
    \label{fig:graph-family-2}
\end{figure}

We conclude the section by introducing a second family of spaces, all of which have singular distance matrix. Still, in almost every case, the spaces in this family carry a microscopic weighting.

\begin{example}
    \label{eg:seven-points}
    For \(0 < a \le 1\) and \(0 < b \le 2\) let \(Q(a, b)\) be the seven-point metric space with distance matrix
    \[
        D_Q^{(a, b)}
        \coloneq
        \begin{pmatrix}
            0 & a & 2a & a & 1 & 1 & 1 \\
            a & 0 & a & 2a & 1 & 1 & 1 \\
            2a & a & 0 & a & 1 & 1 & 1 \\
            a & 2a & a & 0 & 1 & 1 & 1 \\
            1 & 1 & 1 & 1 & 0 & b & b \\
            1 & 1 & 1 & 1 & b & 0 & b \\
            1 & 1 & 1 & 1 & b & b & 0 \\
        \end{pmatrix}.
    \]
    Notice that the matrix \(D_Q^{(a,b)}\) always has determinant zero, since the sum of rows one and three is equal to the sum of rows two and four. Its concentration is given by
         \begin{align*}
            \con*{D_Q^{(a, b)}} &= 
            \begin{cases}
                1
                &
                \text{if } (a, b) = (1, 3/2)
                \\[0.5em]
                \dfrac{3 - 2ab}{6 - 3a - 2b}
                &
                \text{otherwise,}
            \end{cases}
        \end{align*}
    so is infinite whenever \(3a+2b = 6\) and \((a,b) \neq (1,3/2)\).
    
    Let \(\Xi_Q \coloneq (0,1] \times (0, 2]\) denote the parameter space, depicted in \Cref{fig:graph-family-2}. Denote the dotted red line in the figure by \(C_3 \coloneq \{(a,b) \mid 6 = 3a + 2b\} \cap \Xi_Q\) and the dashed blue curve by \(C_4 \coloneq \{(a,b) \mid 2ab = 3\} \cap \Xi_Q\). 
        \begin{itemize}
            \item Since the distance matrix is singular for every space \(Q(a,b)\), \Cref{cor:special-existence} does not apply to any space in this family. However, in the generic part of the parameter space, \(\Xi_Q \setminus (C_3 \cup C_4)\), the distance matrix \(D = D_Q^{(a,b)}\) satisfies \(\con{D} \ne 0, \infty\) and \(\det((D^{\odot 2})|_{\ker(D)}) \neq 0\), so each of these spaces carries a microscopic weighting given by the formula in \Cref{thm:general-existence}. In fact, it turns out that the vector \(\vec{u}'\) vanishes for all these spaces, so for \((a,b) \in \Xi_Q \setminus (C_3 \cup C_4)\) the microscopic weighting is the Moore--Penrose gauging:
            \[
                \micro_{Q(a,b)} 
                =
                \frac{1}{4(3a + 2b - 6)}
                \bigl(
                    2b-3, 2b-3, 2b-3, 2b-3,  
                    4(a-1), 4(a-1), 4(a-1)
                \bigr)^\transpose.
            \]
            \item Every space lying on \(C_3 \setminus \{Q_\mathrm{E}\}\) has \(\con{D} = \infty\), so \Cref{prop:magD_zero_microweighting_does_not_exist} tells us that none of these spaces admits a microscopic weighting.

            \item Every space lying on \(C_4 \setminus \{Q_\mathrm{E}\}\) has \(\con{D} = 0\), so \Cref{thm:general-existence} does not apply here, but each of these spaces does admit a microscopic weighting. Calculating the limit in SageMath, one finds that for \((a,b) \in C_4 \setminus \{Q_\mathrm{E}\}\), the microscopic weighting on \(Q(a,b)\) is
            \[
                \micro_{(a,b)} 
                = 
                \tfrac{1}{1-a}
                \bigl(
                    \tfrac{1}{4}, \tfrac{1}{4}, \tfrac{1}{4}, \tfrac{1}{4}, 
                    -\tfrac{a}{3}, -\tfrac{a}{3}, -\tfrac{a}{3}
                \bigr)^\transpose.
            \]
            Again, for these spaces we do not know how to compute \(\micro_{(a,b)}\) directly from the distance matrix.

            \item The point \(Q_\mathrm{E} = C_3 \cap C_4\) corresponds to the space \(Q(1,3/2)\). Its distance matrix satisfies \(\con{D} = 1\) and \(\det((D^{\odot 2})|_{\ker(D)}) \neq 0\), so \Cref{thm:general-existence} tells us that it admits the microscopic weighting
            \begin{align*}
                \micro^Q_{(1,3/2)}
                &= 
                D^\pseudo\onevector - \vec{u}' \\
                & = 
                \bigl(
                    \tfrac{1}{7}, \tfrac{1}{7}, \tfrac{1}{7}, \tfrac{1}{7}, 
                    \tfrac{1}{7}, \tfrac{1}{7}, \tfrac{1}{7}
                \bigr)^\transpose
                - 
                \bigl(
                    \tfrac{1}{56}, \tfrac{1}{56}, \tfrac{1}{56}, \tfrac{1}{56}, 
                    -\tfrac{1}{42}, -\tfrac{1}{42}, -\tfrac{1}{42}
                \bigr)^\transpose \\
                & =
                \bigl(
                    \tfrac{1}{8}, \tfrac{1}{8}, \tfrac{1}{8}, \tfrac{1}{8},
                    \tfrac{1}{6}, \tfrac{1}{6}, \tfrac{1}{6}
                \bigr)^\transpose.
            \end{align*}
        \end{itemize}
    The spaces lying on \(C_3 \setminus \{Q_E\}\) also serve to illustrate that there exist distance matrices \(D\) with \(\det(D) = 0\) and \(\con{D} = \infty\).
\end{example}

\begin{remark}  
\Cref{thm:general-existence} leaves many questions open concerning the existence of microscopic weightings and their computation. We mention just two.
    \begin{enumerate}
        \item Suppose \(X\) satisfies the conditions of \Cref{thm:general-existence}, and that \(\con{D} \neq \infty\). If \(\det(D) \neq 0\), then the microscopic weighting \(\micro\) on \(X\) is the Moore--Penrose gauging \(D^\pseudo \onevector\), but if \(\det(D) = 0\) then we may or may not have \(\micro = D^\pseudo \onevector\). Is there a simple criterion for when this holds, i.e.~when \(\vec{u}' = \vec{0}\)? 
        \item Can one give a formula for the microscopic weighting in the case that $\con{D} = 0$?
    \end{enumerate}
\end{remark}


\section{The microscopic weighting on a space of strictly negative type}
\label{sec:negative_type}

From this point onward, we focus our attention on metric spaces of \emph{negative type} and \emph{strictly negative type}. In \Cref{sec:negative_type_spaces} we recall the definitions and a standard criterion for a space to be of negative or strictly negative type, and apply this to our family of examples. In \Cref{sec:strict_negative_type_existence}, we prove that every finite metric space of strictly negative type admits a microscopic weighting.


\subsection{Metric spaces of negative type}
\label{sec:negative_type_spaces}

First, we need to recall what it means for a matrix to be conditionally negative semi-definite or definite.

\begin{definition}
\label{def:cond-neg-semi-def}
A symmetric real matrix \(A\) is \define{conditionally negative semi-definite} if it satisfies \(\vec{u}^\transpose A \vec{u} \leq 0\) for every vector \(\vec{u}\) such that \(\onevector^\transpose \vec{u} = 0\). If, moreover, \(\vec{u}^\transpose A \vec{u} = 0\) only for \(\vec{u} = \vec{0}\), then \(A\) is \define{conditionally negative definite}.
\end{definition}

We will denote the hyperplane \(\left\{ \vec{u} \in \R^{n} \mid \onevector^\transpose \vec{u} = 0\right\} \subset \R^n\) by \(\measmasszero{n}\), and refer to its elements as \define{mass zero} vectors.

\begin{definition}
\label{def:negative_type}
A finite metric space \(X\) is of \define{negative type} if its distance matrix \(D\) is conditionally negative semi-definite. The space \(X\) is of \define{strictly negative type} if \(D\) is conditionally negative definite. In general, a metric space \(Y\) is said to be of {negative type} if every finite subset of \(Y\) is of negative type, and of {strictly negative type} if every finite subset is of strictly negative type.
\end{definition}

There is a standard criterion for when a symmetric real matrix is conditionally negative semi-definite or definite. The hyperplane \(\measmasszero{n} \subset \mathbb{R}^n\) is, of course, non-canonically linearly isomorphic to \(\R^{n-1}\). Indeed, for each \(k = 1, \dots, n\) one obtains an isomorphism \(\phi^k \colon \measmasszero{n} \isomto \R^{n-1}\) by forgetting the \(k^\th\) coordinate; the inverse to \(\phi^k\) is the map \(\theta^k \colon \R^{n-1} \to \measmasszero{n}\) given by 
\[
    \theta^k(z_1, \dots, z_{n-1}) = 
    \bigl(z_1, \dots, z_{k-1}, -\textstyle\sum_{\ell =1}^{n-1} z_{\ell}, z_k, \dots, z_{n-1}\bigr).
\] 
Each of the maps \(\theta^k\) can be thought of as a parametrization of \(\measmasszero{n}\). Now, given any symmetric \(n \times n\) real matrix \(A\), the quadratic form that \(A\) determines on \(\measmasszero{n}\) gives rise, via each parametrization \(\theta^k\), to a quadratic form on \(\R^{n-1}\).

\begin{definition}
Let \(A\) be a symmetric \(n \times n\) real matrix. The \define{\(k^\th\) excess matrix} associated to \(A\), denoted \(E^k\), is the \((n-1) \times (n-1)\) matrix of the quadratic form on \(\R^{n-1}\) induced from \(A\) via the parametrization \(\theta^k \colon \R^{n-1} \isomto \measmasszero{n}\).  That is,
    \[
        \vec{z}^\transpose E^k \vec{z} = \theta^k(\vec{z})^\transpose A \theta^k(\vec{z}) 
        \quad
        \text{for all }
        \vec{z} \in \R^{n-1}.
    \]
\end{definition}

A direct calculation using $\theta^1$ gives a formula for the first excess matrix.

\begin{lemma}
\label{lem:excess-formula}
    The first excess matrix of \(A\) has entries
    \(E^1_{\ell, m} = A_{\ell + 1, m + 1} - A_{1, m + 1} - A_{\ell + 1, 1}\). \qed
\end{lemma}

The next statement follows immediately from the construction of the excess matrices and the definition of conditional negative (semi-)definiteness.

\begin{lemma}[{\cite[Lemma~3.5]{HLMT1998}}]
    \label{lem:negative-type-criterion}
    A symmetric \(n \times n\) real matrix \(A\) is conditionally negative semi-definite if and only if, for any (and hence all) \(k = 1, \dots, n-1\), the excess matrix \(E^k\) is positive semi-definite. The matrix \(A\) is conditionally negative definite if and only if, for any (and hence all) \(k\), the excess matrix \(E^k\) is positive definite. Moreover, if \(A\) is conditionally negative semi-definite, then it is conditionally negative definite if and only if for any (and hence all) \(k\) the excess matrix \(E^k\) is invertible.
\end{lemma}

Using the criterion of \Cref{lem:negative-type-criterion}, we can identify the spaces of negative or strictly negative type among the family of spaces introduced in \Cref{eg:five-points}. In \Cref{fig:graph-family}, the shaded area contains the spaces of strictly negative type, while spaces lying on the red dotted line are of negative but not strictly negative type.

\begin{prop}
\label{prop:PQ-negative-type}
The metric space \(P(a, b)\) is of negative type if and only if \(4a + 3b \le 12\). It is of strictly negative type if and only if \(4a + 3b < 12\).
\end{prop}

\begin{proof}
    By \Cref{lem:excess-formula}, the first excess matrix of the distance matrix of \(P(a,b)\) is
    \[E^1(a,b) \coloneq 
    \begin{pmatrix}
        2a & a & a & a \\
        a & 2a & a & a \\
        a & a & 2 & -b + 2 \\
        a & a & -b + 2 & 2 \\
    \end{pmatrix}.
    \]
    By Sylvester's criterion, this matrix is positive definite if and only if the leading principal minors are all positive, and positive semi-definite if and only if all the principal minors (not only the leading ones) are non-negative. Applying the first criterion to the leading principal minors \(2a\), \(3a^2\), \(2(3- a)a^2\) and \((12 - 4a - 3b)a^2b\) establishes that \(E^1(a,b)\) is positive definite if and only if \(4a + 3b < 12\). Applying the second criterion to these and to the additional principal minors \(2\), \((4 - a)a\), \((4 - b)b\) and \(2(4 - a - b)ab\) establishes that \(E^1(a,b)\) is positive semi-definite if and only if \(4a + 3b \leq 12\). The statement now follows from \Cref{lem:negative-type-criterion}.
\end{proof}

A computation similar to that in \Cref{prop:PQ-negative-type} establishes that \(Q(a, b)\) is of negative type if and only if \(3a + 2b \le 6\). As \(\det(D_Q^{(a,b)})=0\) for all \((a,b)\), it is never of strictly negative type.

Historically, metric spaces of negative type derive their significance from work of Schoenberg, as we will outline in \Cref{subsec:schoenberg}. In more recent literature their role tends to be as a convenient class of metric spaces which share many of the good properties of Euclidean space. For example, the \(\ell_p\)-metric on \(\R^n\) is of negative type for each \(p \in [1,2]\). Real and complex hyperbolic space are of negative type, as is every sphere with the geodesic metric \cite[Theorem~3.6]{MeckesPositive2013}. Indeed, Euclidean and hyperbolic spaces are of strictly negative type, as is every finite tree \cite{HLMT1998, HKM2002}.

In the study of magnitude, spaces of negative type are particularly important because, as Meckes proves in~\cite[Theorem~3.3]{MeckesPositive2013}, a finite metric space \(X\) is of negative type if and only if the similarity matrix \(Z(t)\) is positive definite for every \(t > 0\). This guarantees, among other things, that the magnitude function \(\mu_X\) is finite-valued and smooth.


\subsection{Every finite space of strictly negative type admits a microscopic weighting}
\label{sec:strict_negative_type_existence}

To establish the existence of microscopic weightings for spaces of strictly negative type, we will show that these spaces are characterized, among spaces of negative type, by the property that the distance matrix is invertible and has finite concentration. This is a consequence of the following fact, proved by Hjorth, Lison\u{e}k, Markvorsen and Thomassen.

\begin{prop}[{\cite[Proof of Lemma~3.6]{HLMT1998}}]
\label{prop:det_nonzero_strictly_NT}
    Let \(X\) be a finite metric space of negative type with distance matrix \(D\). Then \(X\) is of strictly negative type if and only if the following matrix is invertible:
    \begin{equation*}
    \label{eq:N_X}
    N \coloneq \begin{pmatrix}
	0 & \onevector^\transpose \\
	\onevector & D
    \end{pmatrix}.
    \end{equation*}
\end{prop}

The next two lemmas will allow us to rephrase \Cref{prop:det_nonzero_strictly_NT} in terms of the concentration of \(D\).

\begin{lemma}
    \label{lemma:neg-type-mag-finite}
    Let \(X\) be a finite metric space of negative type with distance matrix \(D\).
    \begin{enumerate}
    \item If \(X\) has just one point, then \(\con{D} = 0\). \label{lemma:neg-type-mag-finite1}
    \item If \(X\) has at least two points, then \(\con{D} \in (0,\infty]\). \label{lemma:neg-type-mag-finite2}
    \item If \(X\) has at least two points and is of strictly negative type, then \(\con{D} \in (0, \infty)\). \label{lemma:neg-type-mag-finite3}
    \end{enumerate}
\end{lemma}

\begin{proof}
    The distance matrix of the one-point space is \(\begin{pmatrix} 0 \end{pmatrix}\), and the vector \(\begin{pmatrix} 1 \end{pmatrix}\) is a gauging for this matrix, with concentration 0. This proves the first claim.
    
    For claim \eqref{lemma:neg-type-mag-finite2}, assume that \(X\) has \(n \geq 2\) points and suppose that \(\con{D} \ne \infty\). Then there exists a gauging for \(D\): some \(\vec{v} \in \R^n\) such that  \(D\vec{v} = \con{D}\onevector\) and \(\onevector^\transpose \vec{v} = 1\). We want to show that \(\con{D} > 0\). For this, observe first that \(D\) must have at least one positive eigenvalue. Indeed, as \(D\) is real and symmetric it has \(n\) real eigenvalues, and as it is not the zero matrix, at least one must be non-zero. Their sum is the trace of \(D\), which is zero as all the diagonal entries are zero.  Thus \(D\) has at least one positive and at least one negative eigenvalue. Let \(\lambda > 0\) be an eigenvalue for \(D\), and \(\vec{u}\) be the corresponding eigenvector. Then \(\vec{u}^\transpose D \vec{u} = \lambda \left|\vec{u}\right|^2 > 0\).  As \(X\) is of negative type this means that \(\onevector^\transpose \vec{u} \ne 0\). Thus, rescaling \(\vec{u}\) if necessary, we can assume that \(\onevector^\transpose \vec{u} = 1\). Now
    \(
        \onevector^\transpose (\vec{v} - \vec{u})
        =
        1 - 1
        =
        0,
    \)
    so, as \(X\) is of negative type, we have
    \begin{align*}
        0 
        &\ge 
        (\vec{v} - \vec{u})^\transpose D (\vec{v} - \vec{u}) 
        =
        \vec{v}^\transpose D\vec{v}
        -
        2 \vec{v}^\transpose D \vec{u}
        +
        \vec{u}^\transpose D \vec{u} \\
       & =
        \con{D} - 2 \con{D} + \vec{u}^\transpose D \vec{u}
        = -\con{D} + \lambda \left|\vec{u}\right|^2.
    \end{align*}
    Thus \(\con{D} \ge \lambda \left|\vec{u}\right|^2 > 0\), as required.
    
    For \eqref{lemma:neg-type-mag-finite3}, we argue by contradiction. Suppose \(\con{D} = \infty\); this is equivalent to \(\matmag{D} = 0\), so there exists \(\vec{w} \neq \vec{0}\) such that \(D\vec{w} = \onevector\) and \(\onevector^\transpose \vec{w} = 0\).  This  \(\vec{w}\) is a mass-zero vector satisfying \(\vec{w}^\transpose D \vec{w} = 0\), so \(X\) cannot be of strictly negative type. 
\end{proof}

In the following statement, \(N\) is the matrix defined in \Cref{prop:det_nonzero_strictly_NT}.

\begin{lemma}
    \label{lem:det-N-adj-D}
    For every finite metric space \(X\) of negative type,
    \(
        \det(N) = - \onevector^\transpose \adj(D) \onevector
    \).
    If \(X\) has at least two points, then
    \(
    \det(N) = -\det(D) / \con{D}
    \).
\end{lemma}
 
\begin{proof}
    Let \(D_{(\hat\imath)}\) denote the matrix \(D\) with its \(i^\th\) row removed, and let \(D_{(\hat\imath\hat\jmath)}\) denote \(D\) with its \(i^\th\) row and \(j^\th\) column removed. By Laplace expansion of the determinant first along the left-most column and then along the top row, we find that
    \begin{align*}
        \det(N)
        &=
        \sum_{i=1}^n (-1)^i 
        \det
        \begin{pmatrix}
            \onevector^\transpose \\ D_{(\hat\imath)}
        \end{pmatrix}
        =
        \sum_{i=1}^n (-1)^i 
        \sum_{j=1}^n (-1)^{j+1} 
        \det(D_{(\hat\imath \hat\jmath)})
        =
        -\sum_{i,j=1}^n
        (-1)^{i+j} 
        \det(D_{(\hat\imath \hat\jmath)}).
    \end{align*}
    Here the  \((i,j)\)-summand is precisely the \((i,j)^\th\) cofactor of \(D\), or, in other words, the \((i,j)\)-entry of the adjugate matrix \(\adj(D)\). This proves the first part of the statement. Now, suppose \(X\) has at least two points. Then, by \Cref{lemma:neg-type-mag-finite}, the distance matrix has non-zero concentration and so
    \(\onevector^\transpose \adj(D) \onevector = \det(D)\matmag{D} = {\det(D)}/{\con{D}}\). This proves the second part.
\end{proof}

In light of \Cref{lem:det-N-adj-D}, \Cref{prop:det_nonzero_strictly_NT} can be rephrased as follows.

\begin{prop}
\label{prop:strict-neg-type-iff-mag-det}
A finite metric space \(X\) of negative type is of strictly negative type if and only if its distance matrix \(D\) is invertible and satisfies \(\con{D} \neq \infty\). \qed
\end{prop}

Together with \Cref{cor:special-existence}, this establishes our most concrete result concerning the existence of microscopic weightings.

\begin{theorem}
\label{thm:strict_negative_type_microwtg}
Every finite metric space \(X\) of strictly negative type admits a microscopic weighting and therefore has the one-point property. The microscopic weighting is given by
\[\lim_{t \to 0} \vec{w}(t) = \con{D} D^{-1} \onevector\]
where \(D\) is the distance matrix of \(X\). In particular, these statements hold for every finite subset of Euclidean space, every finite subset of hyperbolic space, and every finite tree. \qed
\end{theorem}


\section{Interpreting the microscopic weighting}
\label{sec:geometric_interpretations}

We have seen that the microscopic weighting on a finite metric space---when it exists---is always a gauging for the distance matrix (\Cref{thm:microwtg_consequences}). In the setting of negative type spaces, gaugings for the distance matrix already appear in various guises across the metric geometry literature. In this section, we make use of that fact to provide three complementary interpretations of the microscopic weighting. In \Cref{subsec:schoenberg} we relate it to a classical construction in distance geometry, the \emph{Schoenberg polytope} (or \emph{Schoenberg embedding}). In \Cref{subsec:maximum-energy} we show that the microscopic weighting can be interpreted as an optimizing measure for an energy integral determined by the distance function. Then, in \Cref{subsec:curvature}, we observe that in the particular case of a graph equipped with the resistance distance, the microscopic weighting records a certain notion of discrete curvature.


\subsection{Microscopic weightings and the Schoenberg polytope}
\label{subsec:schoenberg}

Classically, spaces of negative type derive their significance from a theorem due to Schoenberg \cite{Schoenberg1935}. Schoenberg was interested in the following problem, suggested by work of Fr\'echet. Given a set of positive real numbers \(\{r_{ij} \mid 1 \leq i < j \leq n\}\), can one determine whether there exist points \(\vec{y}_1,\ldots, \vec{y}_n\) in \(\R^{n}\) such that \(\|\vec{y}_i - \vec{y}_j\| = r_{ij}\) for all \(i < j\)? He proved that such points exist if and only if the matrix of squares \((r_{ij}^2)_{i,j=1}^n\) is conditionally negative semi-definite. In fact, he proved the following sharper statement, paraphrased here to fit the terminology of this paper.

\begin{theorem}[{\cite[Theorem~1]{Schoenberg1935}}]
\label{thm:schoenberg}
Let \(R = \left(r_{ij} \right)_{i,j=1}^n\) be a symmetric real matrix such that \(r_{ii} = 0\) for all \(i\) and \(r_{ij} > 0\) for all \(i \neq j\). The following are equivalent.
\begin{enumerate}
\item There exist points \(\vec{y}_1, \ldots, \vec{y}_n\) in \(\R^n\) such that \(\|\vec{y}_i - \vec{y}_j\| = r_{ij}\) for all \(i < j\).
\item The first excess matrix \(E^1\) of the matrix \(R^{\odot 2} = (r_{ij}^2)_{i,j=1}^n\) is negative semi-definite.
\end{enumerate}
Moreover, if these conditions hold, then the set \(\{\vec{y}_1, \ldots, \vec{y}_n\} \subset \R^{n}\) lies within an affine subspace of dimension \(k = \mathrm{rank}(E^1)\), but not within an affine subspace of dimension \(k-1\).
\end{theorem}

Since the excess matrices of an \(n \times n\) matrix have rank at most \(n-1\), Schoenberg's theorem implies that an \(n\)-point metric space \((X,d)\) is of negative type if and only if the metric space \((X, d^{1/2})\) can be embedded isometrically into \(\R^{n-1}\). In this way one can associate to every finite space \((X, d)\) of negative type a certain convex polytope, namely the convex hull of any such embedding.

\begin{definition}
Given a metric space \(X = (\{x_1, \ldots, x_n\},d_X)\) of negative type, let \(\{\vec{y}_1, \ldots, \vec{y}_n\} \subset \R^{n-1}\) be any set of points such that \(\|\vec{y}_i - \vec{y}_j\|^2 = d_X(x_i,x_j)\). The \define{Schoenberg polytope} associated to \(X\) is the convex hull of the set \(\{\vec{y}_1, \ldots, \vec{y}_n\}\). It is defined up to Euclidean isometries. We will denote it by \(\Delta(X)\).
\end{definition}

The last sentence in the statement of \Cref{thm:schoenberg} implies that an \(n\)-point space \(X\) of negative type is of strictly negative type if and only if the polytope \(\Delta(X)\) has dimension \(n-1\), which is to say that it is a geometric \((n-1)\)-simplex. For \(n \geq 2\), every such simplex has a unique \demph{circumsphere}: an \((n-2)\)-sphere that passes through all its vertices. The centre of that sphere is called the \emph{circumcentre} of the simplex and the radius is its \emph{circumradius}.

If \(X\) is not of strictly negative type, the theorem implies that \(\Delta(X)\) has dimension strictly less than \(n-1\). In that case, its vertices need not lie on a sphere: see \Cref{fig:schoenberg_polytope} and \Cref{eg:schoenberg_polytope} for an example drawn from our family of five-point spaces. In fact, the existence of a sphere containing the vertices of \(\Delta(X)\) is {equivalent} to the existence of a gauging for the distance matrix of \(X\). This is proved by Tarazaga, Hayden and Wells.

\begin{theorem}[{\cite[Theorem~3.4]{TarazagaHaydenWells}}]
\label{thm:circum-euclidean_distance_matrix}
    Given a set of points \(\vec{y}_1, \ldots, \vec{y}_n\) in \(\R^n\), let \(A \in 
    \R^{n \times n}\) be the matrix with entries \(A_{ij} = \|\vec{y}_i - \vec{y}_j\|^2\). The points \(\vec{y}_1,\ldots,\vec{y}_n\) lie on a sphere if and only if there exist \(\vec{u} \in \R^n\) and \(c \in \R\) such that \(A\vec{u} = c \onevector\) and \(\onevector^\transpose \vec{u} = 1\).
\end{theorem}

Concretely, if the matrix \(A\) in \Cref{thm:circum-euclidean_distance_matrix} has a gauging \(\vec{u}\) of concentration \(c\), then the points \(\vec{y}_1, \ldots, \vec{y}_n\) lie on the sphere of radius \(\sqrt{c/2}\) centered at \(\sum_{i=1}^n u_i \vec{y}_i\). This provides our first geometric characterization of the microscopic weighting.

\begin{theorem}
\label{thm:neg_type_schoenberg_sphere}
Let \(X\) be a finite metric space of negative type, and suppose \(X\) admits a microscopic weighting, \(\micro\). Then, given any embedding of \(\Delta(X)\) into \(\mathbb{R}^{n-1}\), the vertices \(\vec{y}_1, \ldots, \vec{y}_n\) lie on the sphere centered at the point \(\vec{p} = \sum_{i=1}^n \widehat{w}_i \vec{y}_i\) whose radius is \(R = \sqrt{\frac{1}{2} \mu_X'(0)}\). 
\qed
\end{theorem}

If \(X\) is not of strictly negative type, then the equation \(\vec{p} = \sum_{i=1}^n \widehat{w}_i \vec{y}_i\) is just one of infinitely many ways to express the point \(\vec{p}\) as a linear combination of the vertices \(\vec{y}_1,\ldots, \vec{y}_n\). However, if \(X\) is of strictly negative type, then this expression is unique: it specifies \(\vec{p}\) in \emph{barycentric coordinates} relative to the simplex \(\Delta(X)\).

\begin{cor}
\label{cor:strict_neg_type_schoenberg_sphere}
If \(X\) is a finite space of strictly negative type, then \(\micro\) specifies the circumcentre of \(\Delta(X)\) in barycentric coordinates, and the circumradius is 
\(R=\sqrt{\frac{1}{2} \mu_X'(0)}\). \qed
\end{cor}

For illustration, we will return to the family of spaces introduced in \Cref{eg:five-points}. In order to draw the associated polytopes, we first specify an embedding of \(\Delta(X)\) into \(\R^{n-1}\). This is  the embedding described in Schoenberg's proof; it also appears in the modern statistics literature, in the context of multidimensional scaling, under the name \emph{double centering}---see, for example, \cite[\S12.1]{BorgGroenen}.

First, given a metric space \(X\) of negative type with distance matrix \(D\), define 
\[G = -\frac{1}{2} \left(I - \frac{1}{n} \onematrix\right)^\transpose D \left(I - \frac{1}{n} \onematrix\right).\]
As \(G\) is real and symmetric, it can be decomposed as \(G = Q \Lambda Q^\transpose\) where the columns of \(Q\) are orthonormal eigenvectors for \(G\) and \(\Lambda\) is the diagonal matrix of eigenvalues. The key to Schoenberg's theorem is that \(D\) is conditionally negative semi-definite if and only if \(G\) is positive definite. In that case its eigenvalues are non-negative and we can write \(G = Y^\transpose Y\) where \(Y^\transpose = Q \Lambda^{\odot 1/2}\). (That is, \(G\) is the \emph{Gram matrix} of the columns of \(Y\).) One can check that the columns \(\vec{y}_1,\ldots, \vec{y}_n\) of \(Y\) satisfy \(\|\vec{y}_i - \vec{y}_j\|^2 = D_{ij}\). Thus, the convex hull of \(\{\vec{y}_1,\ldots, \vec{y}_n\}\) is one embedding of \(\Delta(X)\) into \(\R^{n-1}\).

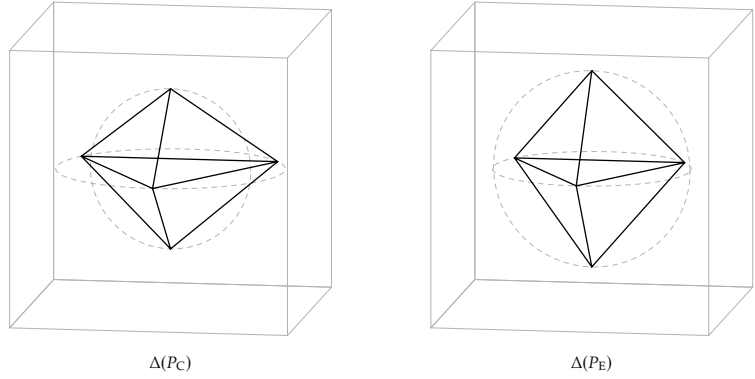
\begin{figure}
    \centering
\tdplotsetmaincoords{80}{99}

\begin{tikzpicture}[tdplot_main_coords, scale=3, line join=round, line cap=round]

  \def\r{.816} 
  \def\h{.577} 
  \def\R{\r} 
  \def\H{\h} 
  \def\s{1} 

  \coordinate (A) at (\r,0,0);
  \coordinate (B) at (-0.5*\r,  0.8660254*\r, 0);
  \coordinate (C) at (-0.5*\r, -0.8660254*\r, 0);

  \coordinate (U) at (0,0,\h);   
  \coordinate (L) at (0,0,-\h);  

  \coordinate (P1) at (-\s,-\s,-\s);
  \coordinate (P2) at ( \s,-\s,-\s);
  \coordinate (P3) at ( \s, \s,-\s);
  \coordinate (P4) at (-\s, \s,-\s);
  \coordinate (P5) at (-\s,-\s, \s);
  \coordinate (P6) at ( \s,-\s, \s);
  \coordinate (P7) at ( \s, \s, \s);
  \coordinate (P8) at (-\s, \s, \s);

  \draw[gray!60]
    (P1)--(P2)--(P3)--(P4)--cycle
    (P5)--(P6)--(P7)--(P8)--cycle
    (P1)--(P5) (P2)--(P6) (P3)--(P7) (P4)--(P8);

  \draw[gray!60] (P1) -- (P2);
  \draw[gray!60] (P1) -- (P4);
  \draw[gray!60] (P1) -- (P5);

  \draw[dashed, gray!60, samples=200, domain=0:360]
    plot ({\R*cos(\x)}, {\R*sin(\x)}, {0});
  \draw[dashed, gray!60, samples=200, domain=0:360]
    plot ({0}, {\H*cos(\x)}, {\H*sin(\x)});

  \draw[black, thick]
    (A)--(B)--(C)--cycle        
    (U)--(A) (U)--(B) (U)--(C)  
    (L)--(A) (L)--(B) (L)--(C); 

  
    \node at (0,0,-1.4) {$\Delta(P_{\mathrm{C}})$};

\end{tikzpicture}
    \quad\quad\quad
\tdplotsetmaincoords{80}{99}

\begin{tikzpicture}[tdplot_main_coords, scale=3, line join=round, line cap=round]

  \def\R{.707}   
  \def\r{\R}
  \def\h{\R}
  \def\s{1} 

  \coordinate (A) at (\r,0,0);
  \coordinate (B) at (-0.5*\r,  0.8660254*\r, 0);
  \coordinate (C) at (-0.5*\r, -0.8660254*\r, 0);
  \coordinate (U) at (0,0,\h);
  \coordinate (L) at (0,0,-\h);

  \coordinate (P1) at (-\s,-\s,-\s);
  \coordinate (P2) at ( \s,-\s,-\s);
  \coordinate (P3) at ( \s, \s,-\s);
  \coordinate (P4) at (-\s, \s,-\s);
  \coordinate (P5) at (-\s,-\s, \s);
  \coordinate (P6) at ( \s,-\s, \s);
  \coordinate (P7) at ( \s, \s, \s);
  \coordinate (P8) at (-\s, \s, \s);

  \draw[gray!60]
    (P1)--(P2)--(P3)--(P4)--cycle
    (P5)--(P6)--(P7)--(P8)--cycle
    (P1)--(P5) (P2)--(P6) (P3)--(P7) (P4)--(P8);

  \draw[gray!60] (P1) -- (P2);
  \draw[gray!60] (P1) -- (P4);
  \draw[gray!60] (P1) -- (P5);

  \draw[dashed, gray!60, samples=200, domain=0:360]
    plot ({\R*cos(\x)}, {\R*sin(\x)}, {0});
  \draw[dashed, gray!60, samples=200, domain=0:360]
    plot ({0}, {\R*cos(\x)}, {\R*sin(\x)});

  \draw[black, thick]
    (A)--(B)--(C)--cycle        
    (U)--(A) (U)--(B) (U)--(C)  
    (L)--(A) (L)--(B) (L)--(C); 
    
    \node at (0,0,-1.4) {$\Delta(P_{\mathrm{E}})$};

\end{tikzpicture}
    \caption{
    The Schoenberg polytopes associated to two members of the family of spaces in \Cref{eg:five-points}. The vertices of \(\Delta(P_{\mathrm{E}})\) lie on a 2-sphere, while those of \(\Delta(P_{\mathrm{C}})\) do not. See \Cref{eg:schoenberg_polytope}.
    }
    \label{fig:schoenberg_polytope}
\end{figure}

\begin{example}
\label{eg:schoenberg_polytope}
\Cref{prop:PQ-negative-type} tells us that the space \(P(a, b)\) is of negative type whenever \(4a + 3b \le 12\). In that case, in the notation of the previous paragraph,
\[Y^\transpose = 
\begin{pmatrix}
    \frac{1}{15}  \sqrt{3} \sqrt{-4  a - 3  b + 12}
    & \frac{1}{2}  \sqrt{a}
    & -\frac{1}{6}  \sqrt{3} \sqrt{a}
    & 0 & 0 \\
    \frac{1}{15}  \sqrt{3} \sqrt{-4  a - 3  b + 12}
    & 0
    & \frac{1}{3}  \sqrt{3} \sqrt{a}
    & 0 & 0\\
    \frac{1}{15}  \sqrt{3} \sqrt{-4  a - 3  b + 12}
    & -\frac{1}{2}  \sqrt{a}
    & -\frac{1}{6}  \sqrt{3} \sqrt{a}
    & 0 & 0\\
    -\frac{1}{10}  \sqrt{3} \sqrt{-4  a - 3  b + 12}
    & 0
    & 0
    & \frac{1}{2}  \sqrt{b} & 0\\
    -\frac{1}{10}  \sqrt{3} \sqrt{-4  a - 3  b + 12}
    & 0
    & 0
    & -\frac{1}{2}  \sqrt{b} & 0
\end{pmatrix}.
\]
The rows of \(Y^\transpose\) give the embedding into \(\R^5\) of the points \(\vec{y}_1,\ldots, \vec{y}_5\). As the last column is zero we see that they embed into \(\R^4 \subset \R^5\), as expected. When \(4a + 3b < 12\) the space \(P(a,b)\) is of strictly negative type, so \(\Delta(P(a,b))\) is 4-dimensional. But when \(4a + 3b = 12\), the space \(P(a,b)\) is \emph{not} of strictly negative type: the first column of \(Y^\transpose\) also vanishes, so \(\Delta(P(a,b))\) embeds into \(\R^3 \subset \R^5\) and we can draw it. 

\Cref{fig:schoenberg_polytope} shows two examples: the Schoenberg polytopes associated to the spaces \(P_{\mathrm{C}} = P(2,4/3)\) and \(P_{\mathrm{E}} = P(3/2,2)\) that are found at the end-points of the red dotted line in \Cref{fig:graph-family}. Notice that the vertices of \(\Delta(P_{\mathrm{E}})\) lie on a 2-sphere while those of \(\Delta(P_{\mathrm{C}})\) do not. Indeed, from \Cref{eg:five-points} we know that \(P_{\mathrm{E}}\) is the only space on the red line for which the distance matrix has a gauging, and it follows from \Cref{thm:circum-euclidean_distance_matrix} that it is the only space on the red line for which the vertices of the Schoenberg polytope lie on a sphere. Travelling from \(P_{\mathrm{C}}\) to \(P_{\mathrm{E}}\) along the red line in \Cref{fig:graph-family}, the family of polytopes \(\{\Delta(P(a,b)) \mid 4a+3b=12\}\) interpolates between \(\Delta(P_{\mathrm{C}})\) and \(\Delta(P_{\mathrm{E}})\) by `squeezing at the middle' until the two circles indicated in \Cref{fig:schoenberg_polytope} are great circles of the same sphere.
\end{example}

\begin{remark}
    This is not the only way that Schoenberg's theorem can be used to study the magnitude function of a finite metric space of negative type: other connections have been explored by Devriendt~\cite{Devriendt2025} and by Asao and Gomi~\cite{AsaoGomi2025}. For any \(n\)-point metric space \(X\) whose similarity matrix \(Z\) is positive definite, Devriendt \cite[\S2.2]{Devriendt2025} constructs what he calls a \emph{similarity embedding} \(\varphi \colon X \to \R^{n-1}\) of the points of \(X\) into Euclidean space, distinct from the Schoenberg embedding. He shows \cite[Theorem 2.12]{Devriendt2025} that the magnitude \(\mu_X(1) = \matmag{Z}\) can be expressed in terms of the circumradius \(r\) of the similarity embedding, via the formula
    \begin{equation}
    \label{eq:devriendt-thm}
        \mu_X(1) = \frac{1}{1-2r^2}.
    \end{equation}
    The similarity embedding and related ideas are also investigated by Asao and Gomi in their paper \cite{AsaoGomi2025}, which appeared at almost exactly the same time as Devriendt's.
    
    For a finite space \(X\) of negative type, the similarity matrix of \(tX\) is positive definite for every \(t > 0\). Thus, for every \(t > 0\), one has a \emph{scale-\(t\) similarity embedding} \(\varphi_t \colon X \to \R^{n-1}\) and \eqref{eq:devriendt-thm} yields the equation
    \begin{equation}
    \label{eq:devriendt-thm2}
    \mu_X(t) = \frac{1}{1-2r(t)^2}
    \end{equation}
    where \(r(t)\) is the circumradius of \(\varphi_t(X)\). Meanwhile, our \Cref{thm:neg_type_schoenberg_sphere} gives the identity
    \[\mu_X'(0) = 2R^2\]
    where \(R\) is the circumradius of the classical Schoenberg embedding \(\Delta(X)\). Presumably this identity is in some sense the derivative at \(t=0\) of equation \eqref{eq:devriendt-thm2}; in this paper, we refrain from making this precise.
\end{remark}


\subsection{Microscopic weightings as energy-maximizing measures}
\label{subsec:maximum-energy}

Suppose \(X\) is a finite metric space of negative type that admits a microscopic weighting. \Cref{thm:neg_type_schoenberg_sphere} offers a characterization of the microscopic weighting that is satisfying from the perspective of classical distance geometry, but one might hope for an interpretation that is more intrinsic to the geometry of \(X\). To that end, we now show that the microscopic weighting can be interpreted as an \emph{energy-maximizing measure} on \(X\).

First, let us explain what is meant here by `energy'. Given a measurable space \(Y\), a measurable function \(f \colon Y \times Y \to \R\), and a signed measure \(\nu\) on \(Y\), let
\[E_f(\nu) = \iint_{Y \times Y} f(y, y')\,\dd\nu(y)\dd\nu(y').\]
If one thinks of \(f\) as a \emph{potential function} and \(\nu\) as representing a distribution of charge across \(Y\), then the integral \(E_f(\nu)\) represents the \emph{total potential energy} of \(\nu\) with respect to \(f\). A classical example, when \(Y\) is a subset of \(\R^m\) for \(m > 2\), takes the potential function \(f\) to be the {Newtonian kernel}, defined by \(f(y,y') \coloneq c \|y - y'\|^{2-m}\) for a certain dimensional constant \(c\).

The Newtonian kernel is typical of potentials with physical origins, in that \(f(y,y')\) decays as the distance between \(y\) and \(y'\) increases. Since the mid-twentieth century, however, interest in energy integrals has extended beyond such physical examples. For instance, Bj\"orck~\cite{Bjoerck:Distributions-of-positive-mass} takes \(Y\) to be a subset of Euclidean space and \(f\) to be a positive power of the distance function: \(f(y, y') \coloneq \|y - y'\|^\alpha\) with \(\alpha > 0\). In contrast to more traditional energy integrals---for example, with \(\alpha <0\)---one is interested in this case in maximizing the integral over some space of measures of fixed mass, rather than minimizing it. The intuition in both cases is that the optimizing measures represent maximally `spread-out' distributions of mass, or charge, across \(Y\). 

Nickolas and Wolf~\cite{Nickolas-Wolf:distance-geometry-I, Nickolas-Wolf:distance-geometry-II, Nickolas-Wolf:distance-geometry-III}, following Alexander and Solarsky~\cite{Alexander-Stolarsky:Extremal-problems}, extend these ideas to metric spaces that are not Euclidean, but fix the exponent \(\alpha = 1\). They investigate a `geometric constant' \(M(X)\) that is defined as follows.

\begin{definition}
\label{def:d-energy-infinite}
    Given a compact metric space \(Y\), let \(\measmassone(Y)\) denote the set of finite, signed, Borel measures of mass 1 on \(Y\). For \(\nu \in \measmassone(Y)\) define the \define{energy} of \(\nu\) to be
    \[
        I(\nu) \coloneq E_d(\nu) = \iint_{Y \times Y} d(y, y') \, \dd \nu(y) \,\dd \nu(y').
    \]
    Denote by \(M(Y)\) the \define{maximal energy} of \(Y\)---that is,
    \[
        M(Y) \coloneq
        \sup\{I(\nu) \mid \nu \in \measmassone(Y)\}.
    \]
    If \(\nu \in \measmassone(Y)\) satisfies \(I(\nu) = M(Y)\), we say that \(\nu\) is an \define{energy-maximizing measure}.  If a sequence \(\nu_1, \nu_2, \ldots \in \measmassone(Y)\) satisfies \(\lim_{q \to \infty} I(\nu_q) = M(Y)\), it is said to be an \define{energy-maximizing family of measures}.
\end{definition}

Nickolas and Wolf are interested almost exclusively in compact metric spaces of negative type (which are \demph{quasihypermetric spaces} in their terminology). The following fact, which they prove in the generality of compact spaces, explains why they restrict their attention to this class.

\begin{prop}[{\cite[Theorem~3.1]{Nickolas-Wolf:distance-geometry-I}}]
    \label{thm:non-negative-type-unbound-energy}
    If \(X\) is a finite metric space that is not of negative type, then \(M(X) = \infty\). In particular, no energy-maximizing measure exists on \(X\).
\end{prop}

\begin{proof}
    For an \(n\)-point metric space, \(\measmassone(X)\) can be identified with \(\{\vec{u} \in \R^n \mid \onevector^\transpose \vec{u} = 1\}\). If \(X\) is not of negative type, there is some \(\vec{u} \in \R^n\) with \(\onevector^\transpose \vec{u} = 0\) and \(I(\vec{u}) > 0\).  So, for any \(\vec{v} \in \measmassone(X)\) and \(q \in \N\), we have \(\vec{v} + q \vec{u} \in \measmassone(X)\). Since \(I(\vec{v} + q \vec{u}) \to \infty\) as \(q \to \infty\), this shows that \(M(X) = \infty\).
\end{proof}

We are going to see that when a finite metric space of negative type admits a microscopic weighting, it is energy-maximizing. Again, the proof rests on the fact that a microscopic weighting is a gauging for the distance matrix. That gaugings are energy-maximizing measures is an observation that goes back to Alexander and Solarsky for subsets of Euclidean space \cite[Theorem~3.3]{Alexander-Stolarsky:Extremal-problems}; it is extended to compact spaces of negative type by Nickolas and Wolf \cite[Theorem~3.1]{Nickolas-Wolf:distance-geometry-II}. (In their terminology, a gauging is a \demph{\(d\)-invariant measure} and its concentration is its \demph{value}.) In the case of finite spaces we can give a concise proof, so we include it for clarity.

\begin{prop}
\label{prop:maximal-distance-energy}
    Let \(X\) be a finite metric space of negative type, with distance matrix \(D\).
    \begin{enumerate}
        \item \label{part:maximal-distance-energy-2} A measure \(\vec{v}\in \measmassone(X)\) is energy-maximizing if and only if \(\vec{v}\) is a gauging for \(D\).
        \item \label{part:maximal-distance-energy-3} We have \(M(X) = \con{D}\).
    \end{enumerate}
\end{prop}

\begin{proof}
    For \eqref{part:maximal-distance-energy-2}, we first show that an energy-maximizing measure is a gauging for \(D\). This is a standard application of a Lagrange multiplier. For the problem of maximizing \(\vec{u}^\transpose D \vec{u}\) subject to \(\onevector^\transpose \vec{u} = 1\), define the Lagrangian for \(\vec{u} \in \R^n\) and \(\gamma \in \R\) as
    \[
        \lagrangian(\vec{u}, \gamma) 
        \coloneq
        \vec{u}^\transpose D \vec{u} + \gamma (1 - \onevector^\transpose \vec{u}).
    \]
    Then we have 
    \[
        \nabla \lagrangian(\vec{u}, \gamma)
        =
        (\vec{u}^\transpose D - \gamma \onevector^\transpose, 1 - \onevector^\transpose \vec{u}).
    \]
    Lagrangian theory says that if \(\vec{v}\) is maximizing then there is a \(c \in \R\) such that \(\nabla\lagrangian(\vec{v}, c) = \zerovector^\transpose\).  That is,
    \(
        \vec{v}^\transpose D = c \onevector^\transpose
    \)
    and
    \(
        \onevector^\transpose \vec{v} = 1
    \),
    so \(\vec{v}\) is a gauging for \(D\).

    Next we show that a gauging is a maximal measure. This is the point at which we need \(X\) to be of negative type.  Let \(\vec{v}\) be a gauging, so \(D\vec{v} = \con{D} \onevector\) and \(\onevector^\transpose \vec{v} = 1\).  For \(\vec{u} \in \measmassone(X)\) we need to show that \(I(\vec{v}) \ge I(\vec{u})\).
    We have that \(\onevector^\transpose(\vec{v} - \vec{u}) = 0\), so, as \(X\) is of negative type, we also have \(I(\vec{v} - \vec{u}) \le 0\).  Expanding that out gives \(2\vec{v}^\transpose D \vec{u} \ge I(\vec{v}) + I(\vec{u})\) and hence
    \begin{align*}
        2I(\vec{v})
        &=
        2 \vec{v}^\transpose D \vec{v} 
        = 
        2 \con{D} 
        = 
        2 \con{D} \onevector^\transpose \vec{u}
        =
        2 \vec{v}^\transpose D \vec{u}
        \ge 
        I(\vec{v}) + I(\vec{u}).
    \end{align*}
    Thus \(I(\vec{v}) \ge I(\vec{u})\) and so \(\vec{v}\) is maximal, as required.

    For \eqref{part:maximal-distance-energy-3}, note that \(\con{D} = \infty\) if and only if \(\matmag{D} = 0\), in which case there exists \(\vec{w} \in \R^n\) such that \(D\vec{w} = \onevector\) and \(\onevector^\transpose \vec{w} = 0\).  Pick \(\vec{u} \in \measmassone(X)\).  Then for \(q \in \N\), we have \(\vec{u} + q \vec{w} \in \measmassone(X)\) and 
    \[I(\vec{u} + q \vec{w}) = \vec{u}^\transpose D \vec{u} + 2 q \vec{u}^\transpose D \vec{w} + q^2 \vec{w}^\transpose D\vec{w} = I(u) + 2q \to \infty\]
    as \(q \to \infty\). Thus, if \(\con{D} = \infty\) then \(M(X) =\infty\).
    If \(\con{D} \ne \infty\) then a gauging \(\vec{v}\) exists, and by \eqref{part:maximal-distance-energy-2} we have 
    \(M(X) = I(\vec{v}) = \vec{v}^\transpose D \vec{v} = \vec{v}^\transpose (\con{D} \onevector) = \con{D}\).
\end{proof}

As we have seen in \Cref{eg:five-points}, there are finite metric spaces of negative type for which the distance matrix has no gauging and hence has infinite concentration. For instance, this is the case for the space \(P_{\mathrm{C}} = P(2,4/3)\). However, by \Cref{thm:microwtg_consequences}, if a finite metric space admits a microscopic weighting \(\micro\), then its distance matrix does have a gauging, namely \(\micro\), and its concentration is the derivative at zero of the magnitude function. In particular, that holds for every space of strictly negative type by \Cref{thm:strict_negative_type_microwtg}. Combining those facts with \Cref{prop:maximal-distance-energy}, we arrive at the following statement.

\begin{theorem}
\label{thm:microwtg_energy_maximizing}
Let \(X\) be a finite metric space of negative type, and suppose \(X\) admits a microscopic weighting, \(\micro\). Then \(\micro\) maximizes the energy integral \(I\) on \(X\), and
\[M(X) = \mu'_X(0).\]
In particular, this holds for every finite metric space of strictly negative type. \qed
\end{theorem}


\subsection{Microscopic weightings and discrete curvature}
\label{subsec:curvature}

\begin{figure}
   \centering
    \raisebox{-.45\height}{\includegraphics[width=0.45\textwidth]{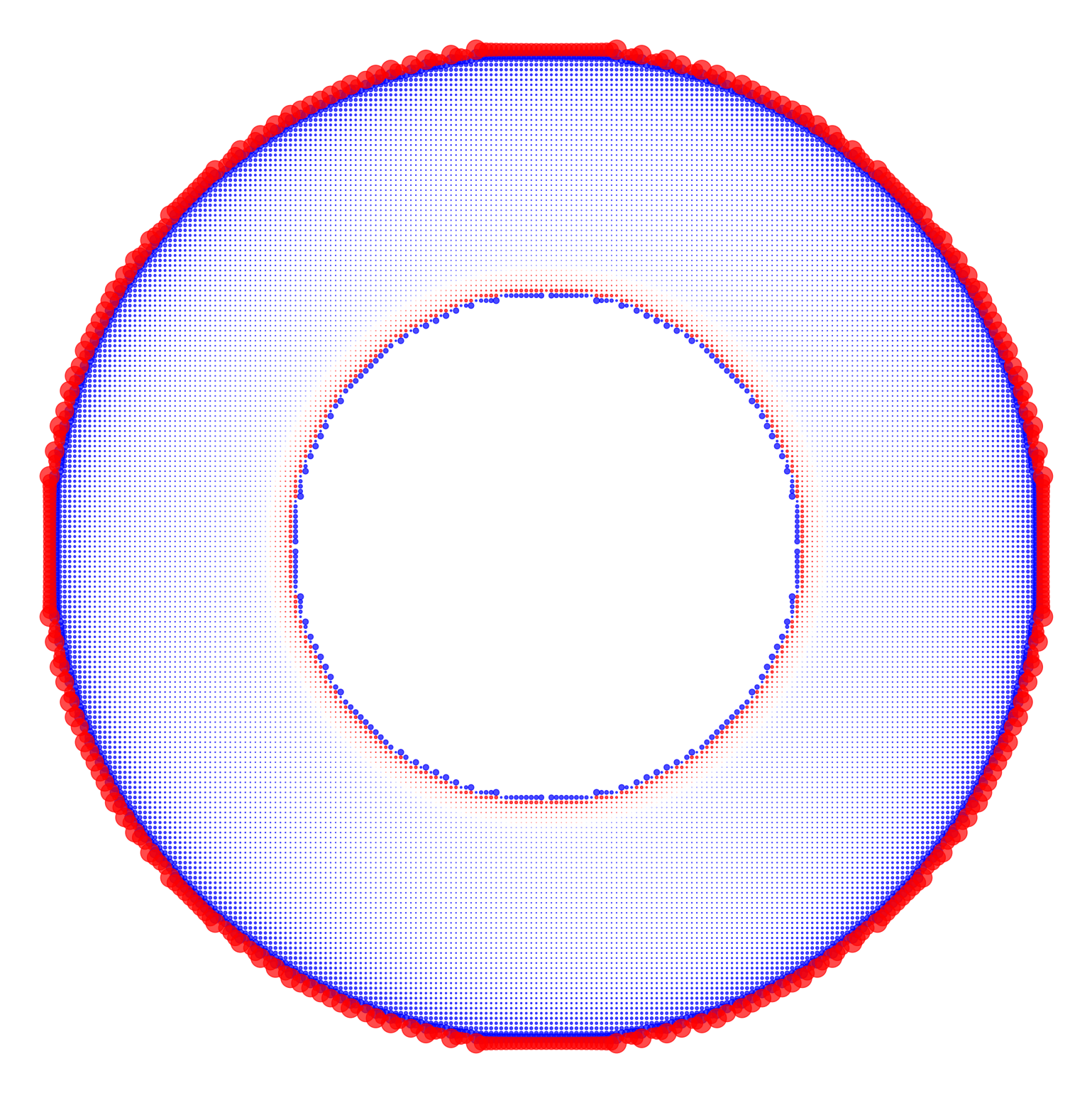}}
   \quad
   \raisebox{-.45\height}{
   \includegraphics[width=0.45\textwidth]{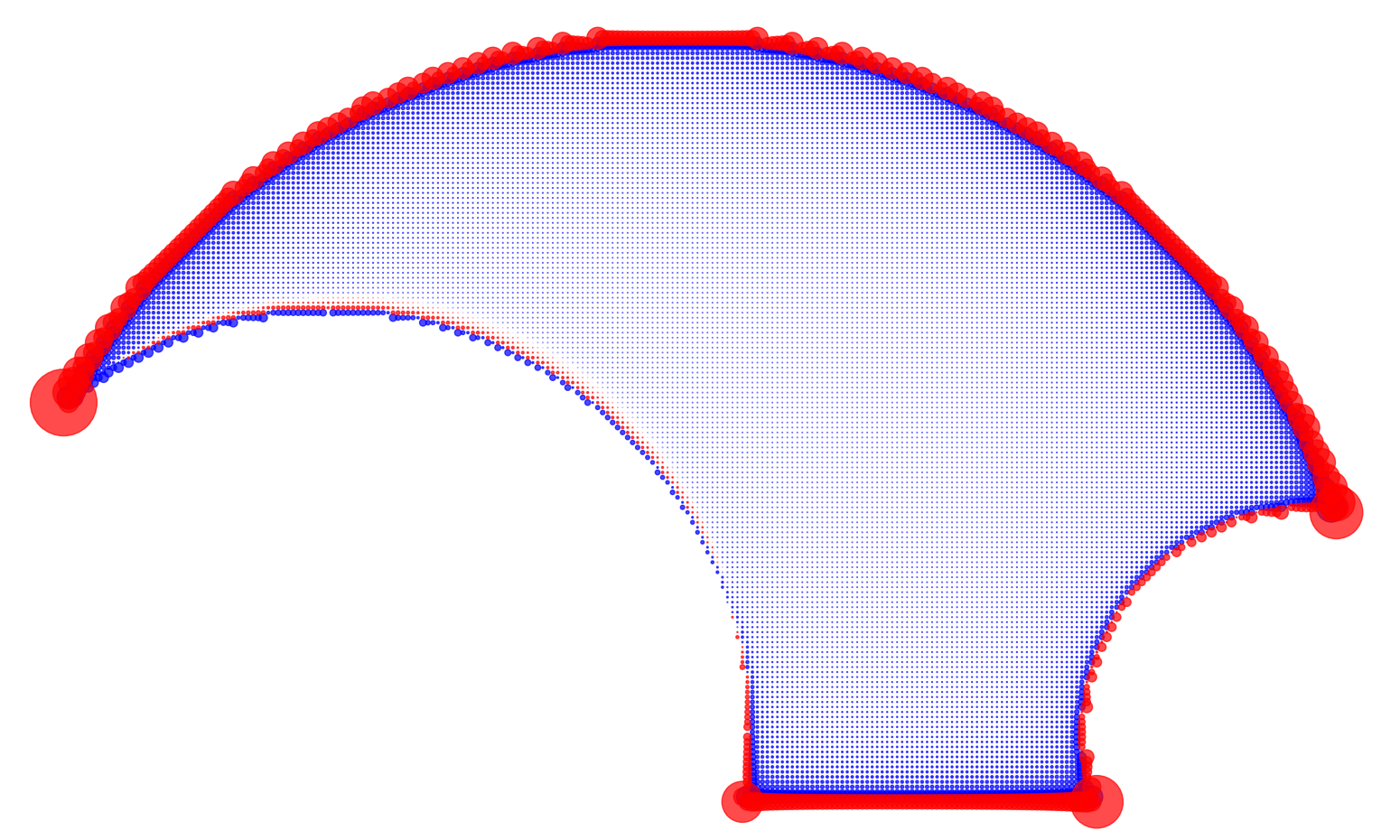}}
   \caption{The microscopic weightings on grid approximations to two regions in the plane. The area of the disc at each point is proportional to the weight at that point. Red represents a positive weight, blue a negative weight. Notice that the largest weights appear to be concentrated on the boundary.}
   \label{fig:microweightings-subsets-of-plane}
\end{figure}

Our third interpretation relates the microscopic weighting to a notion of discrete curvature for graphs, introduced recently by Devriendt and Lambiotte \cite{DevriendtLambiotte}. The definition depends on the notion of \emph{resistance distance}, which is a metric on the vertex set of a graph in which the distance between two vertices represents the effective resistance between two nodes in a corresponding electrical network (see Klein and Randi\'{c} \cite{KleinRandic1993}). How the resistance distance is actually defined will not be important here. What is relevant is that it is always of strictly negative type \cite[Proposition~4.9]{Devriendt-thesis}, so its distance matrix is invertible (\Cref{prop:strict-neg-type-iff-mag-det}), which allows Devriendt and Lambiotte to make the following definition.

\begin{definition}[{\cite[Definition~1 and Appendix~B2]{DevriendtLambiotte}}]
\label{def:resistance-curvature}
Let \(G\) be a graph with vertices \(v_1,\ldots, v_n\), and let \(\Omega\) be the distance matrix for the resistance distance on \(G\). Let \(\kappa \in \R^n\) be the unique vector that satisfies 
\begin{equation}
\label{eq:resistance-curvature}
\Omega \kappa = \frac{1}{\kappa^\transpose \Omega \kappa} \onevector.
\end{equation}
The \define{resistance curvature} at each vertex \(v_i\) of \(G\) is the number \(\kappa_i \in \R\).
\end{definition}

Devriendt and Lambiotte make the case that this is a reasonable notion of discrete curvature, which captures interesting information about the combinatorics of \(G\) and its geometry under the resistance distance. All that we will add is the observation that the microscopic weighting for the resistance distance on a graph is the vector that records its resistance curvature. Since `the magnitude of a graph' usually refers to its magnitude with respect to the shortest path metric, we emphasize that in the following statement, \(\vec{w}(t)\) is the unique weighting for the matrix \(\left(\exp(-t\Omega_{v_i,v_j})\right)_{i,j=1}^n\).

\begin{theorem}
\label{thm:microwtg-curvature}
Let \(G\) be any finite graph. If we equip \(G\) with the resistance distance, then it admits the microscopic weighting
\(\micro = \kappa\).
\end{theorem}

\begin{proof}
As the resistance distance is of strictly negative type, \Cref{thm:strict_negative_type_microwtg} tells us that \(G\) admits a microscopic weighting, \(\micro\). By \Cref{thm:microwtg_consequences}, \(\micro\) is the unique gauging for \(\Omega\); to see that it satisfies \eqref{eq:resistance-curvature} we need to see that \(\con{\Omega} = 1/(\micro^\transpose \Omega \micro)\). Indeed, by definition of gaugings we have \(\con{\Omega} \onevector^\transpose = \micro^\transpose \Omega\) and \(\onevector^\transpose \micro = 1\), so \(1 = \frac{1}{\con{\Omega}} \micro^\transpose \Omega \micro\) and hence \(\con{\Omega} = 1/\micro^\transpose \Omega \micro\). Thus, by \Cref{def:resistance-curvature}, we have \(\kappa = \micro\).
\end{proof}

The examples in \Cref{fig:microweightings-subsets-of-plane} suggest that, in the Euclidean setting, any relationship between microscopic weightings and curvature will not be straightforward.

\begin{remark}
    Shortly after Devriendt and Lambiotte's paper~\cite{DevriendtLambiotte} appeared, Steinerberger~\cite{Steinerberger2022} proposed another notion of discrete curvature that is similar but distinct. Expressing Steinerberger's definition in the terminology of this paper, a \emph{curvature vector} for a graph \(G\) on \(n\) vertices is a vector \(\frac{1}{n} \vec{w}\) where \(\vec{w} \in\R^n\) is a weighting for the shortest-path distance matrix of \(G\). In a more recent preprint~\cite{Steinerberger2026}, Steinerberger broadens his scope to encompass all compact metric spaces, but restricts his attention to probability measures rather than signed measures. A \emph{curvature measure} for a finite metric space, in the sense of~\cite{Steinerberger2026}, is a gauging for the distance matrix whose entries are all non-negative.
    
    Since Steinerberger's curvature measures are probability measures, one cannot expect to relate them directly to the microscopic weighting in typical cases. Instead, we conjecture that Steinerberger's curvature measure on a finite or compact metric space \(X\) can be obtained by taking a small-\(t\) limit of \emph{diversity-maximizing} probability measures on the scaled spaces \(tX\), which are introduced for finite metric spaces by Leinster and Meckes~\cite{LeinsterMeckes2016} and for compact metric spaces by Leinster and Roff~\cite{LeinsterRoff}.
\end{remark}


\section{Towards microscopic weightings on compact metric spaces}
\label{sec:infinite-spaces}

For a finite metric space \(X\) of strictly negative type, we have proved the following statements in \Cref{prop:maximal-distance-energy} and \Cref{thm:microwtg_energy_maximizing}. First, that the derivative at zero of the magnitude function is equal to the maximal energy of \(X\), as defined by Nickolas and Wolf:
\begin{equation}
\label{finite-statement1}
    \mu_X'(0)= M(X).
\end{equation}
Second, that \(X\) admits a microscopic weighting \(\micro\), which is a gauging for the distance matrix with concentration equal to the maximal energy:
\begin{equation}
\label{finite-statement2}
    \onevector^\transpose \micro = 1 \quad \text{ and } \quad D \micro = M(X) \onevector.
\end{equation}
Third, that \(\micro\) is an energy-maximizing measure:
\begin{equation}
\label{finite-statement3}
    I(\micro) = \sup_{\vec{v} \in \measmassone(X)} I(\vec{v}).
\end{equation}
It seems reasonable to wonder whether analogues of these statements can be proved for infinite spaces of strictly negative type---or perhaps, more specifically, for compact subsets of Euclidean space. This section explores that idea.

Statement \eqref{finite-statement1} is easy to translate into a conjecture about compact spaces of strictly negative type, because it is already known how to define both the magnitude function and the maximal energy in that generality. In \Cref{subsec:max-energy-derivative-infinite} we will observe that the equality \(\mu_X'(0)= M(X)\) holds non-trivially for odd-dimensional balls in Euclidean space, and conjecture that it holds in general for compact metric spaces of strictly negative type.

It is less straightforward to write down a conjecture generalizing statements \eqref{finite-statement2} and \eqref{finite-statement3}, because the notion of weighting is more subtle for compact spaces than for finite spaces, which makes it more difficult to define microscopic weightings in that generality. For subsets of Euclidean space, however, we can attempt it. In \Cref{sec:microwtg-infinite} we propose a definition of microscopic weighting, using Meckes' theorem that every subset of Euclidean space carries a unique weighting, which in general is not a measure but a {distribution} in the sense of Schwartz.

In \Cref{sec:three-ball} we study in more detail the example of \(B^3\), the three-dimensional ball in Euclidean space. We show that \(B^3\) admits a microscopic weighting, and that an analogue of statement \eqref{finite-statement2} holds in this case. Though it is known that there exists no energy-maximizing measure on \(B^3\), we suggest that the microscopic weighting may be interpretable as the limit of a energy-maximizing family of measures---hence, that it may deserve to be understood as an \emph{energy-maximizing distribution} on \(B^3\).


\subsection{Maximal energy and the derivative of the magnitude function}
\label{subsec:max-energy-derivative-infinite}

If \(X\) is a finite metric space of negative type, so are all its subsets; moreover, one has \(\mu_V(t) \leq \mu_X(t)\) for every subset \(V \subseteq X\) and every \(t > 0\) \cite[Corollary~2.4.4]{LeinsterMagnitude2013}. This guarantees that the magnitude function of \(X\) satisfies
\[\mu_X(t) = \sup\{\mu_V(t) \mid V \subset X\}\]
for every \(t \in (0, \infty)\). Magnitude can therefore be extended from finite to compact spaces of negative type by taking a supremum over finite subsets.

\begin{definition}
\label{def:mag-infinite}
Let \(Y\) be a compact metric space of negative type.
The \define{magnitude function} of \(Y\) is the function \(\mu_Y \colon (0, \infty) \to \R \cup \{\infty\}\) defined by
\[\mu_Y(t) = \sup \{\mu_X(t) \mid X \subset Y \text{ is finite}\}.\]
\end{definition}

There are not many compact metric spaces whose magnitude functions are known explicitly, but the unit balls in each odd-dimensional Euclidean space are among them. We will denote the unit ball in dimension \(m \geq 1\) by \(B^m\). Barcel\'o and Carbery proved~\cite[Theorem~4]{BarceloCarbery:magnitude-of-compact-sets} that \(\mu_{B^{2p+1}}\) is a rational function, and Willerton~\cite{Willerton:odd-balls} gave an explicit combinatorial formula for it, together with SageMath code for doing computations. 

In private communication with Mark Meckes in 2018, Willerton conjectured a formula for the derivative of the magnitude at zero of \(B^{2p+1}\) and suggested a strategy for proving it using his combinatorial formula. Meckes completed the proof~\cite[proof of Theorem~4]{MeckesMagnitude2019}. Meanwhile, the maximal energy of \(B^{2p+1}\) is also known: following work of Alexander~\cite{Alexander:two-problems} on the three-dimensional unit ball, Hinrichs, Nickolas and Wolf calculated the maximal distance energy for the unit ball in any finite-dimensional Euclidean space~\cite[Theorem~2.1]{HinrichsNickolasWolf:metric-geometry-of-unit-ball}. Putting these two results together, we get the following.

    \begin{theorem}
        \label{theorem:derivative-maximal-energy-balls}
        For each odd-dimensional ball, the derivative at zero of the magnitude function coincides with the maximal energy. That is, for \(p = 0, 1, 2, \dots\), we have
        \[
            \mu_{B^{2p+1}}'(0) 
            = 
            \prod_{i=1}^{p}\frac{2i}{2i-1}
            =
            M(B^{2p+1}).
        \]
    \end{theorem}

This encourages us to make the following conjecture.
\begin{conj}
    \label{conj:derivative-maximal-energy-neg-type}
    For every compact metric space \(Y\) of strictly negative type, we have
    \(
        \mu_Y'(0)
        =
        M(Y)
    \).
\end{conj}

To have a hope of proving \Cref{conj:derivative-maximal-energy-neg-type} by methods extending those in this paper, one would want to understand microscopic weightings, gaugings and concentration in the context of compact spaces. We will take a first step in that direction by defining microscopic weightings for compact subsets of Euclidean space.

\begin{remark}
    Meckes conjectured~\cite[Conjecture~5]{MeckesMagnitude2019} that for \(K\) a non-empty convex subset of Euclidean space one should have \(\mu_K'(0) = \frac{1}{2} V_1(K)\) where \(V_1(K)\) is the first intrinsic volume of \(K\). This conjecture is based partly on the fact that in the case of the odd-dimensional ball, \(\frac{1}{2}V_1(B^{2p+1})\) has the same formula as given in \Cref{theorem:derivative-maximal-energy-balls}.  In recent private communication Meckes has indicated that one can show \(M(K) \le \frac{1}{2} V_1(K)\) for \(K\) a convex body in Euclidean space, so \Cref{conj:derivative-maximal-energy-neg-type} is consistent with Theorem~1.6 of his paper~\cite{MeckesMagnitude2019}.
\end{remark}


\subsection{Defining microscopic weightings for compact subsets of Euclidean space}
\label{sec:microwtg-infinite}

First, let us return temporarily to the finite setting, where we can rephrase the definition of the magnitude function as follows. Given a finite metric space \(X\) with points \(x_1,\ldots, x_n\), we will continue to think of \(\R^n\) as the space of finite, signed measures on \(X\). For each \(t \in (0, \infty)\), the similarity matrix \(Z(t)\) defines a symmetric bilinear form \(\Zpair{-}{-}_t\) on \(\R^n\) by
\begin{equation}\label{Zform_finite}
\Zpair{\vec{v}}{\vec{v}'}_t = \vec{v}^\transpose Z(t) \vec{v}'.
\end{equation}
If \(X\) is of negative type, the matrix \(Z(t)\) is positive definite for every \(t\) \cite[Theorem~3.3]{MeckesPositive2013}. In that case \(\Zpair{-}{-}_t\) is an inner product and induces a norm, \(\|-\|_t\). If \(\vec{w}(t) \in \R^n\) is a weighting for \(Z(t)\), we can write
\[\mu_X(t) = \matmag{Z(t)} = \onevector^\transpose \vec{w}(t) =  (Z(t) \vec{w}^\transpose) \vec{w}(t) = \vec{w}(t)^\transpose Z(t) \vec{w}(t) = \Zpair{\vec{w}(t)}{\vec{w}(t)}_t,\]
or, if \(X\) of negative type, \(\mu_X(t) = \|\vec{w}(t)\|_t^2\).

Meckes observed that this approach can be used to extend the notion of weightings beyond finite spaces \cite{Meckes:magnitude-diversity-capacity-dimension}. Given \emph{any} metric space \(U\), the expression in (\ref{Zform_finite}) can be adapted to define a symmetric bilinear form on the space \(\finsuppmeas(U)\) of {finitely supported}, finite, signed measures on \(U\). For each \(t \in (0, \infty)\) and \(\nu, \nu' \in \finsuppmeas(U)\), let
\begin{equation}
\label{wtg_form_real}
\Zpair{\nu}{\nu'}_t = \iint_{u,u' \in U} e^{-td(u,u')} \dd\nu(u) \dd\nu'(u').
\end{equation}
If \(U\) is of negative type then for every \(t\) the symmetric bilinear form defined in this way is positive definite, so it makes \(\finsuppmeas(U)\) into an inner product space \cite[\S3]{Meckes:magnitude-diversity-capacity-dimension}. Meanwhile, for each \(\nu \in \finsuppmeas(U)\), the function \(Z_t\nu \colon U \to \R\) defined by
\[Z_t\nu(u) = \int_{u' \in U} e^{-td(u,u')} \dd \nu(u')\]
belongs to a certain Banach space \(\tilde{C}(U)\) of functions, and the map \(Z_t \colon \finsuppmeas(U) \to \tilde{C}(U)\) is an injective, bounded, linear operator. It therefore extends uniquely to an injective linear operator on the completion of \(\finsuppmeas(U)\) \cite[Proposition~3.2]{Meckes:magnitude-diversity-capacity-dimension}. Thus, one can make the following definition.

\begin{definition}[{\cite[Definition~3.3]{Meckes:magnitude-diversity-capacity-dimension}}]
\label{def:weighting-space}
Let \(U\) be any metric space of negative type. The \textbf{weighting space} of \(U\) at scale \(t \in (0, \infty)\) is the completion \(\wtspace_t(U)\) of \(\finsuppmeas(U)\) with respect to the inner product \(\Zpair{-}{-}_t\) defined in \eqref{wtg_form_real}. A \textbf{weighting} for \(U\) at scale \(t\) is an element \(w_t \in \wtspace_t(U)\) such that \(Z_tw_t(u) = 1\) for each \(u \in U\).
\end{definition}

Note that if a weighting exists at a given scale \(t\), it is necessarily unique: this follows from the injectivity of \(Z_t\). By definition, a weighting \(w_t \in \wtspace_t(U)\) consists of a sequence of finitely supported measures on \(U\). We will also write \(\wtspace(U)\) to denote \(\wtspace_1(U)\). 

Given a compact space \(Y\) of negative type, it makes sense to ask whether the magnitude function \(\mu_Y \colon (0, \infty) \to \R \cup \{\infty\}\), defined as in \Cref{def:mag-infinite}, can be computed using weightings, as in the finite case. Meckes proves that it can be.

\begin{theorem}[{\cite[Theorem~3.4]{Meckes:magnitude-diversity-capacity-dimension}}]
\label{wtg_mag}
Let \(Y\) be a compact metric space of negative type. If \(Y\) has a weighting at scale \(t\), say \(w_t\), then \(\mu_Y(t) = \Zpair{w_t}{w_t}_t\).
\end{theorem}

The theorem confirms that Meckes' notion of weighting is a good one; however, \Cref{def:weighting-space} is not terribly explicit about what type of thing a weighting can be. In the case of a compact subset \(Y\) of Euclidean space, it turns out that more can be said: \(\wtspace_t(Y)\) can be regarded as a space of \emph{distributions} on \(Y\).

Let us recall very briefly that a distribution on an arbitrary subset \(U \subseteq \R^m\) is an element of the topological linear dual \(\mathcal{D}'(U)\) to some topological vector space \(\mathcal{D}(U)\) of functions from \(U\) to \(\R\). That is, a distribution on \(U\) is a linear functional \(u \colon \mathcal{D}(U) \to \R\) with a certain continuity property that we will not specify here. For a textbook introduction to the theory, see Friedlander and Joshi~\cite{Friedlander}. One often writes \(\dispair{u}{f}\) for \(u(f)\). In particular, suitable measures on \(U\) will give rise to distributions via \(\dispair{\nu}{f} = \int_U f \dd \nu\), but not every distribution on \(U\) can be represented in this way.

Meckes identifies the weighting space \(\wtspace(\R^m)\) with the \emph{Bessel potential space} \(\besselspace^{-(m+1)/2}\), whose elements are a certain class of distributions on \(\R^m\) \cite[Proposition~5.1]{Meckes:magnitude-diversity-capacity-dimension}. As Meckes explains, for \(\alpha \geq 0\), the Bessel potential space \(\besselspace^{\alpha}\) is actually a space of functions and the space of distributions \(\besselspace^{-\alpha}\) can be identified with the topological linear dual of \(\besselspace^{\alpha}\). Now, suppose \(Y \subset \R^m\) is compact. Then the weighting space \(\wtspace(Y)\) is the closure of the finitely supported measures \(\finsuppmeas(Y)\) inside \(\wtspace(\R^m)\), so we have \(\wtspace(Y) \subset \wtspace(\R^m) \cong \besselspace^{-(m+1)/2}\), which says that any weighting \(w\) for \(Y\) can be thought of as a distribution. In fact, Meckes has proved that in this setting a weighting will always exist; by the observation following \Cref{def:weighting-space}, it is unique.

\begin{theorem}[{\cite[Corollary~5.3]{Meckes:magnitude-diversity-capacity-dimension}}]
\label{thm:compact-Y-has-weighting}
Let \(Y \subset \R^m\) be compact. Then \(Y\) has a unique weighting, which is an element of \(\besselspace^{-(m+1)/2}\).
\end{theorem}

We can use this to define microscopic weightings for compact subsets of Euclidean space. For every compact subset \(Y \subset \R^m\) and every \(t > 0\), the scaled space \(tY\) is isometric to the space \(Y_t \coloneq \{ t \vec{y} \mid \vec{y} \in Y\}\) regarded as a compact subset of \(\R^m\). Thus, Meckes' results imply that \(tY\) carries a unique weighting, which can be identified with an element of \(\besselspace^{-(m+1)/2}\). This allows us to propose the following definition.

\begin{definition}
\label{def:microwtg-compact}
Let \(Y \subset \R^m\) be compact, and for each \(t > 0\) let \(w_t \in \besselspace^{-(m+1)/2}\) be the unique weighting on the scaled space \(tY\). A \define{microscopic weighting} for \(Y\) is a distribution \(\widehat{w} \in \besselspace^{-(m+1)/2}\) with the property that
\begin{equation}
\label{eq:microwtg-compact}
\dispair{\widehat{w}}{f} = \lim_{t \to 0} \dispair{w_t}{f}
\end{equation}
for every function \(f \in \besselspace^{(m+1)/2}\). That is, \(\widehat{w} \coloneq \lim_{t \to 0} w_t\) in the \emph{weak-* topology} on \(\besselspace^{-(m+1)/2}\).
\end{definition}


\subsection{The three-dimensional ball}
\label{sec:three-ball}

We focus now on the example of \(B^3\), the three-dimensional ball of radius one, seen as a subspace of \(\R^3\). Alexander~\cite{Alexander:two-problems} showed that no energy-maximizing measure exists on \(B^3\): there is no signed measure \(\nu\) with \(\int_{B^3} \dd\nu =1\) such that \(I(\nu)= M(B^3)\). However, he was able to construct a sequence of measures \(\nu_1, \nu_2, \dots\) with \(I(\nu_q) \to 2 = M(B^3)\) as \(q \to \infty\). If we denote by \(\bar{s}_\rho\) the mass-one uniform measure on the radius \(\rho\) two-sphere \(S^2_\rho \subset B^3\), Alexander's energy-maximizing family \((\nu_q)\) is given by
\begin{equation}
\label{eq:Alexander-family}
    \nu_q \coloneq q \bar{s}_1 + (1 - q) \bar{s}_{1-1/q}
    \quad
    \text{for }
    q = 1, 2, \dots.
\end{equation}
We will return to this energy-maximizing family of measures below and interpret its limit as a distribution: the microscopic weighting. For, although it has no energy-maximizing measure, \(B^3\) does have a microscopic weighting, which we can compute.

As we are considering \(B^3\) as a subspace of \(\R^3\), the microscopic weighting belongs to the Bessel potential space \(\besselspace^{-2}\), which is dual to \(\besselspace^{2}\). It will be useful to describe \(\besselspace^{2}\) more concretely: it is the set of square-integrable functions
\[
    \Bigl\{
        f \in L^2(\R^3) 
        \Bigm|
        \tfrac{\partial_1^{\alpha_1}}{\partial x_1^{\alpha_1}} 
        \tfrac{\partial_2^{\alpha_2}}{\partial x_2^{\alpha_2}} 
        \tfrac{\partial_3^{\alpha_3}}{\partial x_3^{\alpha_3}} 
        f \in L^2(\R^3) 
        \text{ for all } 
        (\alpha_1,\alpha_2,\alpha_3) \in \mathbb{N}^3 
        \text{ with } 
        \textstyle\sum\limits_i \alpha_i \leq 2 
    \Bigr\}
    .
\]
Willerton~\cite{Willerton:odd-balls} gave a method to compute the weighting for odd-dimensional balls of arbitrary radius, and we can use that method to obtain the unique weighting for the radius \(t\) three-ball \(B^3_t \subset \R^3\). The space that really interests us is not \(B^3_t\) but \(t B^3\), the unit three-ball with the metric scaled by a factor of \(t\). However, we have the isometry \(B^3_t \to t B^3\) given by \(\vec{b} \mapsto (1/t)\vec{b}\), and using that isometry, along with the results in \cite{Willerton:odd-balls}, one arrives at the following expression.

\begin{lemma}
On \(t B^3\), the unique weighting \(w_t \in \besselspace^{-2}\) is given on a function \(f \in \besselspace^2\) by
\[
    \dispair{w_t}{f}
    =
    \tfrac{1}{3!} t^3 \int_{B^3} f(\vec{b}) \,\dd \bar{b}
    +
    (t^2 + 2t + 1)
    \int_{S^2} f(\vec{s}) \,\dd \bar{s}
    +
    (\tfrac{1}{2}t + 1)
    \int_{S^2} \tfrac{\partial}{\partial \vec{n}} f(\vec{s}) \,\dd \bar{s}
    .
\]
Here, \(\bar{b}\) and \(\bar{s}\) denote the Lebesgue measure restricted to the unit three-ball \(B^3\) and the unit two-sphere \(S^2\) respectively, and normalized so that \(B^3\) and \(S^2\) each have total mass 1. The notation \(\frac{\partial}{\partial \vec{n}} f\) means the normal derivative of \(f\), pointing out of \(B^3\). \qed
\end{lemma}

It is clear from the expression in the lemma that, for every function \(f \in \besselspace^2\), the limit \(\lim_{t \to 0} \dispair{w_t}{f}\) exists. Thus, the three-dimensional ball admits a microscopic weighting, specified as follows. Note in particular that the microscopic weighting is supported entirely on \(S^2\), the boundary of the ball.

\begin{prop}
\label{prop:microwtg-on-B3}
On \(B^3\), the microscopic weighting  \(\widehat{w} \in \besselspace^{-2}\) is given by
\[
    \dispair{\widehat{w}}{f}
    =
    \lim_{t \to 0}
    \dispair{w_t}{f}
    =
    \int_{S^2} f(\vec{s}) \,\dd \bar{s}
    +
    \int_{S^2} \tfrac{\partial}{\partial \vec{n}} f(\vec{s}) \,\dd \bar{s},
\]
for each function \(f \in \besselspace^2\). 
\qed
\end{prop}

To see that the microscopic weighting on a finite space \(X\) of strictly negative type is energy-maximizing, we used the fact that it is a gauging for the distance matrix, with concentration equal to the maximal energy. That is, writing \(\vec{d}_i\) for the \(i^\th\) column of the distance matrix, so \((\vec{d}_i)_j = d(x_j,x_i) = d(x_i,x_j)\), we have 
\begin{equation}
\label{eq:gauging_equation}
    \onevector^\transpose \micro = 1 \quad \text{ and } \quad \vec{d}_i^\transpose \micro = M(X) \text{ for each } i \in \{1,\ldots, n\}.
\end{equation} 
To extend the methods of this paper beyond finite spaces, one might begin by defining gaugings and concentration in the generality of compact spaces. We will not go that far here; instead, we simply prove a statement analogous to \eqref{eq:gauging_equation} for the particular case of the three-dimensional ball.

In formulating the statement, we have to be attentive to the fact that the constant function \(1\) is not an element of \(L^2(\R^3)\), so it does not belong to \(\besselspace^2\). However, one can certainly find compactly supported, smooth functions (`bump functions') that are identically equal to 1 on \(B^3\), and we can choose one of these as a replacement for the vector \(\onevector\). More care must be taken with the Euclidean distance function \(d(\vec{b},-)\), which is not differentiable at \(\vec{b}\). However, because the microscopic weighting on \(B^3\) is supported only on the boundary, we can avoid difficulties in this case by restricting attention to points \(\vec{b}\) lying within the interior. We denote the interior by \(\mathrm{int}(B^3)\).

\begin{theorem}
\label{thm:compact_gauging_maximal_energy}
    Let \(h\) be any function in \(\besselspace^{2}\) such that \(h \equiv 1\) on a neighbourhood of \(B^3\), and for each \(\vec{b} \in \mathrm{int}(B^3)\), choose \(k_{\vec{b}} \in \besselspace^{2}\) such that \(k_{\vec{b}} = d(\vec{b},-)\) on a neighbourhood of \(S^2\). The microscopic weighting \(\widehat{w}\) on \(B^3\) satisfies
    \[\dispair{\widehat{w}}{h} = 1 \quad \text{and} \quad \dispair{\widehat{w}}{k_{\vec{b}}} = M(B^3) \text{ for all } \vec{b} \in \mathrm{int}(B^3).\]
\end{theorem}

\begin{proof}
That \(\dispair{\widehat{w}}{h} = 1\) is clear from the expression for \(\widehat{w}\) in~\Cref{prop:microwtg-on-B3} (recalling that we are using the normalized measure \(\bar s\) on \(S^2\)). To evaluate \(\dispair{\widehat{w}}{k_{\vec{b}}}\) requires the following identities, obtained by direct computation: if \(\vec{b}\) is in the interior of \(B^3\) and \(\vec{s}\) is in \(S^2\), then     
\begin{gather*}
    \int_{\vec{s} \in S^2} \d(\vec{b}, \vec{s}) \,\dd \bar{s} 
    = 
    1 + \tfrac{1}{3}\left|\vec{b}\right|^2;
    \quad
    \tfrac{\partial}{\partial \vec{n}}\d(\vec{b}, \vec{s}) 
    = 
    \tfrac{1 - \vec{b}\cdot \vec{s}}{\sqrt{\left|\vec{b}\right|^2 + 1 - 2\vec{b}\cdot \vec{s}}};
    \\
    \int_{\vec{s} \in S^2} \tfrac{\partial}{\partial \vec{n}}\d(\vec{b}, \vec{s}) \,\dd \bar{s} 
    = 
    1 - \tfrac{1}{3}\left|\vec{b}\right|^2.
\end{gather*}
Using these identities along with \Cref{prop:microwtg-on-B3}, we calculate for each \(\vec{b}\) that
\[
        \dispair{\widehat{w}}{k_{\vec{b}}} 
        =  
        \int_{\vec{s} \in S^2} \d(\vec{b},\vec{s}) \,\dd \bar{s} 
        +
        \int_{\vec{s} \in S^2} \tfrac{\partial}{\partial \vec{n}}\d(\vec{b},\vec{s}) \,\dd \bar{s} 
        =
        2,
    \]
which, by Alexander's result~\cite{Alexander:two-problems}, is the maximal energy of \(B^3\).
\end{proof}

\Cref{thm:compact_gauging_maximal_energy} suggests that the microscopic weighting on the three-dimensional ball can be thought of as something like a `gauging' for the distance function, with `concentration' equal to the maximal energy. In the finite setting, gaugings for the distance function are precisely energy-maximizing measures (\Cref{prop:maximal-distance-energy}), and we used that fact to prove that the microscopic weighting on a finite space of strictly negative type is an energy-maximizing measure (\Cref{thm:microwtg_energy_maximizing}). We cannot hope to make a precisely analogous statement for \(B^3\), because Alexander's work shows that \(B^3\) has no energy-maximizing measure. Instead, we close with a computation that suggests that the microscopic weighting on \(B^3\) may be interpreted as the limit of Alexander's maximizing family in some suitable space of distributions.

\begin{theorem}
\label{thm:alexander_limit}
    Let \(\nu_1, \nu_2, \dots\) be Alexander's energy-maximizing family of measures on \(B^3\), and let  \(\widehat{w}\) be the microscopic weighting of \Cref{prop:microwtg-on-B3}. Let \(f \colon \R^3 \to \R\) be such that \(f\) and \(\tfrac{\partial}{\partial \vec{n}} f\) are both continuous on a neighbourhood of \(B^3\). Then
    \[
        \lim_{q \to \infty} \dispair{\nu_q}{f} = \lim_{q\to \infty} 
        \int_{B^3} f(\vec{b})  \,\dd \nu_q 
        =
        \dispair{\widehat{w}}{f}.
    \]
\end{theorem}

\begin{proof}
Using the definition of \(\nu_q\) in equation \eqref{eq:Alexander-family} and the expression for \(\widehat{w}\) in \Cref{prop:microwtg-on-B3}, we have
    \begin{align}
        \lim_{q\to \infty}
            \int_{B^3} f(\vec{b}) \,\dd \nu_q 
        \nonumber 
        &=
        \lim_{q\to \infty}
        \left(
            \int_{B^3} q f(\vec{b})  \,\dd \bar{s}_1 
            + 
            \int_{B^3} (1 - q)f(\vec{b}) \,\dd\bar{s}_{1-1/q}
        \right)
        \nonumber \\
        &=
        \lim_{q\to \infty}
        \left(
            \int_{S^2} q f(\vec{s})\,\dd \bar{s} 
            + 
            \int_{S^2} (1 - q) f\left(\left(1-\tfrac{1}{q}\right)\vec{s}\right)\,\dd\bar{s}
        \right)
        \label{eq:alexander_limit_1} \\
        &=
        \lim_{q\to \infty}
        \left(
            \int_{S^2} f\left(\left(1-\tfrac{1}{q}\right)\vec{s}\right) \,\dd \bar{s} 
            + 
            \int_{S^2} \frac{f(\vec{s}) - f\left(\left(1-\tfrac{1}{q}\right)\vec{s}\right)}{\tfrac{1}{q}} \,\dd\bar{s}
        \right)
        \label{eq:alexander_limit_2} \\
        &=
        \int_{S^2} \lim_{q\to \infty} f\left(\left(1-\tfrac{1}{q}\right)\vec{s}\right) \,\dd \bar{s} 
        + 
        \int_{S^2} \lim_{q\to \infty} \frac{f(\vec{s}) - f\left(\left(1-\tfrac{1}{q}\right)\vec{s}\right)}{\tfrac{1}{q}} \,\dd\bar{s}
        \label{eq:alexander_limit_3} \\
        &=
        \int_{S^2} f(\vec{s}) \,\dd \bar{s}
        +
        \int_{S^2} \tfrac{\partial}{\partial \vec{n}} f(\vec{s}) \,\dd \bar{s}
        \label{eq:alexander_limit_4} \\
        &=
        \dispair{\widehat{w}}{f}.
        \nonumber 
    \end{align}
Here, the equality \eqref{eq:alexander_limit_1} follows from a change of variables, using the fact that the measure \(\bar s_{1-1/q}\) is supported on an interior concentric sphere. The equality \eqref{eq:alexander_limit_2} is obtained by algebraic manipulation.

The exchange of integrals and limits in \eqref{eq:alexander_limit_3} is possible thanks to the dominated convergence theorem. As \(|f|\) is continuous on \(B^3\) it is bounded there, and if we choose \(g_1\) to be a smooth, compactly supported function that is constantly equal to \(\sup_{B^3}|f|\) on \(S^2\), then \(g_1\) is integrable and satisfies 
\[\bigl|f((1-1/q) \vec{s})\bigr| \leq g_1(\vec{s})\]
for all \(\vec{s} \in S^2\) and \(q 
\geq 1\). Similarly, we can choose \(g_2\) to be a smooth, compactly supported function that is constantly equal to \(\sup_{S^2} \left| \tfrac{\partial}{\partial \vec{n}} f \right| + 1\) on \(S^2\). We have 
\[
    \biggl|  
        \tfrac{1}{q} 
        \Bigl(
            f(\vec{s}) - f\bigl(\left(1 - 1/q \right)\vec{s}\bigr) 
        \Bigr)
    \biggr| 
    - 
    \biggl| 
        \tfrac{\partial}{\partial \vec{n}} f(\vec{s}) 
    \biggr| 
    \leq 
    \Biggl| 
        \biggl|  
            \tfrac{1}{q} 
            \Bigl(
                f(\vec{s}) - 
                f\bigl(
                    (1-1/q)\vec{s}
                \bigr) 
            \Bigr)
        \biggr| 
        - 
        \biggl| 
            \tfrac{\partial}{\partial \vec{n}} f(\vec{s}) 
        \biggr| 
    \Biggr| 
\]
and for \(q \gg 1\) the right hand side is bounded by \(1\), say, by the definition of the normal derivative.  Hence, for all \(\vec{s} \in S^2\) and \(q \gg 1\),
\[
    \biggl|  
        \tfrac{1}{q} 
        \Bigl(
            f(\vec{s}) - f\bigl(\left(1 - 1/q \right)\vec{s}\bigr) 
        \Bigr)
    \biggr| 
    < 
    1
    +
    \biggl| 
        \tfrac{\partial}{\partial \vec{n}} f(\vec{s}) 
    \biggr| 
    \leq 
    g_2(\vec{s})
    .
\]
Equality \eqref{eq:alexander_limit_3} follows from the dominated convergence theorem using \(g_1 + g_2\) as the dominating function.

Finally, the equality \eqref{eq:alexander_limit_4} follows using the continuity of \(f\) and the definition of the normal derivative.
\end{proof}

Extrapolating from these observations, we suggest that in the setting of Euclidean space it may be fruitful to consider the existence of energy-maximizing \emph{distributions} where energy-maximizing \emph{measures} do not exist.


\section{Proof of the existence theorem}
\label{sec:existence-proofs}

In \Cref{sec:microscopic_weightings} we gave only a sketch proof of \Cref{thm:general-existence}, our most general existence theorem for microscopic weightings. Now, we will fill in the gaps. \Cref{subsec:special-existence-proof} contains the proof of \Cref{prop:UL-inverse}, and in \Cref{subsec:general-existence-proof} we apply it to prove the theorem itself.


\subsection{Proof of Proposition \ref{prop:UL-inverse}}
\label{subsec:special-existence-proof}

First, let us record some basic algebraic properties of the `big-O' notation outlined in \Cref{def:big-o}. Given \(f,g \colon (0,\infty) \to \R\) such that \(f(t) = O(t^p)\) and \(g(t) = O(t^q)\) for some \(p,q \in \R\) as \(t \to 0\), it is straightforward to see that
\[(f+g)(t) = O(t^{\min\{p, q\}}) \quad \text{and} \quad (f g)(t) = O(t^{p+q}) \quad \text{as } t \to 0.\]
From this one can derive that if matrices \(A(t) \in \R^{m \times n}\) and \(B(t) \in \R^{n \times r}\) are such that \(A(t) = O(t^p)\) and \(B(t) = O(t^q)\), in the sense of \Cref{def:big-o}, then 
\[A(t)B(t) = O(t^{p+q}) \quad \text{as } t \to 0.\]
If \(A_1(t) = M + A_0(t)\) and \(B_1(t) = N + B_0(t)\) where \(M \in \R^{m \times n}\) and \(N \in \R^{n \times r}\) do not depend on \(t\), then
\begin{align*}
    A_1(t) B_1(t) &= MN + MB_0(t) + A_0(t)N + A_0(t)B_0(t) \\
    & = MN + O(t^q) + O(t^p) + O(t^{p+q}) \quad \text{as } t \to 0,
\end{align*}
and if \(p\) and \(q\) are nonnegative this implies that \(A_1(t)B_1(t) = MN + O (t^{\min\{p,q\}})\). We will be particularly interested to know how these asymptotics interact with matrix inversion.

\begin{lemma}
\label{lem:invert-perturbed}
Fix an invertible matrix \(A \in \R^{n \times n}\) and let \(B(t)\) be an \(n \times n\) matrix which is \(O(t)\) as \(t \to 0\).
\begin{enumerate}
    \item \label{lem:invert-perturbed-pt1} For all sufficiently small \(t > 0\) the matrix \(I + B(t)\) is invertible, and \((I + B(t))^{-1} = I + O(t)\) as \(t \to 0\).
    \item \label{lem:invert-perturbed-pt2} For all sufficiently small \(t > 0\) the matrix \(A + B(t)\) is invertible, and \((A + B(t))^{-1} = A^{-1} + O(t)\) as \(t \to 0.\)
\end{enumerate}
\end{lemma}

\begin{proof}
For part \eqref{lem:invert-perturbed-pt1}, recall that the determinant of a matrix is a polynomial in the entries. By inspection, the diagonal entries of \(I\) contribute the term \(1\) to \(\det(I+B(t))\), while the other monomial terms each involve at least one entry of \(B(t)\), so they are \(O(t)\) as \(t \to 0\). Thus, \(\det(I + B(t)) = 1 + O(t)\), from which it follows that \(I + B(t)\) is invertible for all sufficiently small \(t > 0\).

Each entry of \(\adj(I + B(t))\) is again a determinant; by inspection, it has one term equal to the corresponding entry of \(I\), followed by monomials each involving at least one entry of \(B(t)\). Thus, \(\operatorname{adj}(I + B(t)) = I + O(t)\). It follows that
\[\left(I + B(t)\right)^{-1} = \frac{\adj(I + B(t))}{\det(I + B(t))} = \left(I + O(t)\right)\left(1 + O(t)\right)^{-1} = I + O(t),\]
since \((1 + O(t))^{-1} = 1 + O(t)\) for the scalar factor. 

For part \eqref{lem:invert-perturbed-pt2}, write \(M(t) =  A^{-1} B(t)\), so that \(A + B(t) = A \left(I + M(t)\right)\). Since all the entries of \(B(t)\) are \(O(t)\) as \(t \to 0\), the same is true of \(M(t)\), so part \eqref{lem:invert-perturbed-pt1} tells us that \(\left(I + M(t)\right)^{-1} = I + O(t)\) as \(t \to 0\). Thus,
\[\left(A + B(t)\right)^{-1} = \left(I + M(t)\right)^{-1} A^{-1} = (I + O(t)) A^{-1} = A^{-1} + O(t),
\]
as claimed.
\end{proof}

The next lemma is a fairly direct consequence of a matrix identity due to Sherman and Morrison \cite{ShermanMorrison}. Given an invertible matrix \(A \in \R^{n \times n}\) and any vectors \(\vec{u}, \vec{v} \in \R^n\), the \demph{Sherman--Morrison formula} says that \(A + \vec{u} \vec{v}^\transpose\) is invertible if and only if \(1 + \vec{v}^\transpose A^{-1} \vec{u} \neq 0\), in which case its inverse is
\[\left(A + \vec{u} \vec{v}^\transpose \right)^{-1} = A^{-1} - \frac{A^{-1} \vec{u} \vec{v}^\transpose A^{-1}}{1 + \vec{v}^\transpose A^{-1} \vec{u}}.\]

\begin{lemma}
\label{lem:U-tA}
Let \(D \in \R^{n \times n}\) be invertible. The following statements hold for all \(\vec{v}\) such that \(\vec{v}^\transpose D^{-1} \vec{v} \neq 0\).
\begin{enumerate}
\item \label{lem:U-tA-pt1} For all sufficiently small \(t > 0\), the matrix \(\vec{v}\vec{v}^\transpose - tD\) is invertible.
\item \label{lem:U-tA-pt2} As \(t \to 0\), we have \((\vec{v}\vec{v}^\transpose - tD)^{-1} = O(t^{-1})\).
\item \label{lem:U-tA-pt3} As \(t \to 0\), we have \(\left(\vec{v}\vec{v}^\transpose - tD\right)^{-1} \vec{v} = (\vec{v}^\transpose D^{-1} \vec{v})^{-1} D^{-1} \vec{v}  + O(t)\).
\end{enumerate}
\end{lemma}

\begin{proof}
Let \(\vec{w} = D^{-1} \vec{v}\). Since \(D\) is invertible, for each \(t>0\) we can factor the matrix \(\vec{v}\vec{v}^\transpose - tD\) as
\[
\vec{v}\vec{v}^\transpose - tD = - tD \left( I - t^{-1}D^{-1}\vec{v}\vec{v}^\transpose \right) 
= - tD \left( I - t^{-1}\vec{w}\vec{v}^{\transpose} \right).
\]
By the Sherman--Morrison formula, whenever \(t \neq \vec{v}^\transpose D^{-1} \vec{v}\) the matrix \(I - t^{-1} \vec{w}\vec{v}^\transpose\) is invertible, with
\[
\left( I - t^{-1}\vec{w}\vec{v}^{\transpose} \right)^{-1}
= I + \frac{t^{-1} \vec{w} \vec{v}^\transpose}{1 - t^{-1} \vec{v}^\transpose \vec{w}}
= I + \frac{1}{t - \vec{v}^\transpose D^{-1} \vec{v}} \vec{w} \vec{v}^\transpose. 
\]
If \(\vec{v}^\transpose D^{-1} \vec{v} \neq 0\), then for all sufficiently small \(t > 0\) we have \(t \neq \vec{v}^\transpose D^{-1} \vec{v}\). It follows that for all sufficiently small \(t > 0\) the matrix \(\vec{v}\vec{v}^\transpose - tD\) is invertible, with
\[
(\vec{v}\vec{v}^\transpose - tD)^{-1} = \left( - tD (I - t^{-1}\vec{w} \vec{v}^\transpose) \right)^{-1} = - \frac{1}{t} \left( D^{-1} + \frac{1}{t - \vec{v}^\transpose D^{-1} \vec{v}} \vec{w} \vec{v}^\transpose D^{-1} \right),
\]
which is \(O(t^{-1})\) as \(t \to 0\). This proves parts \eqref{lem:U-tA-pt1} and \eqref{lem:U-tA-pt2}. It also tells us that for sufficiently small \(t > 0\) we have
\begin{align*}
\left(\vec{v}\vec{v}^\transpose - tD\right)^{-1} \vec{v} 
&= - \frac{1}{t} \left(D^{-1} + \frac{1}{t - \vec{v}^\transpose D^{-1} \vec{v}} \vec{w} \vec{v}^\transpose D^{-1} \right) \vec{v} \\
& = -\frac{1}{t} \left( 1 +  \frac{\vec{v}^\transpose D^{-1} \vec{v}}{t - \vec{v}^\transpose D^{-1} \vec{v}} \right) \vec{w} \\
& = \frac{-1}{t - \vec{v}^\transpose D^{-1} \vec{v}} \vec{w}.
\end{align*}
As \(t \to 0\), this gives
\begin{align*}
\left(\vec{v}\vec{v}^\transpose - tD\right)^{-1} \vec{v} &= \frac{1}{\vec{v}^\transpose D^{-1} \vec{v}} \left( \frac{1}{1 - \frac{t}{\vec{v}^\transpose D^{-1} \vec{v}}} \right) \vec{w} \\
& = \frac{1}{\vec{v}^\transpose D^{-1} \vec{v}} \left( 1 + \frac{t}{\vec{v}^\transpose D^{-1} \vec{v}} + O(t^2) \right) \vec{w} \\
& = \frac{1}{\vec{v}^\transpose D^{-1} \vec{v}} \vec{w} + O(t)
\end{align*}
which proves part \eqref{lem:U-tA-pt3}.
\end{proof}

Together, Lemmas \ref{lem:invert-perturbed} and \ref{lem:U-tA} yield the proposition.

\begin{propUL-inverse}
Let \(D \in \R^{n \times n}\) be invertible and let \(B(t) \in\R^{n \times n}\) be \(O(t^2)\) as \(t \to 0\). Given any vector \(\vec{v} \in \R^n\) such that \(\vec{v}^\transpose D^{-1} \vec{v} \neq 0\), the following statements hold.
\begin{enumerate}
    \item For all sufficiently small \(t > 0\), the matrix \(\vec{v}\vec{v}^\transpose - tD + B(t)\) is invertible.
    \item As \(t \to 0\), we have \(\left(\vec{v}\vec{v}^\transpose - tD + B(t) \right)^{-1} = O(t^{-1})\).
    \item As \(t \to 0\), we have \(\left(\vec{v}\vec{v}^\transpose - tD + B(t) \right)^{-1} \vec{v} = (\vec{v}^\transpose D^{-1} \vec{v})^{-1} D^{-1} \vec{v} + O(t)\).
\end{enumerate}
\end{propUL-inverse}

\begin{proof}
Part \eqref{lem:U-tA-pt1} of \Cref{lem:U-tA} tells us there exists \(\epsilon_1 > 0\) such that \(\vec{v}\vec{v}^\transpose - tD\) is invertible for all \(0 < t < \epsilon_1\). For such \(t\), let \(M(t) = (\vec{v}\vec{v}^\transpose - tD)^{-1} B(t)\), so that \(\left(\vec{v}\vec{v}^\transpose - tD + B(t) \right) = (\vec{v}\vec{v}^\transpose - tD)(I + M(t))\). Since the entries of \(B(t)\) are \(O(t^2)\) by assumption and those of \((\vec{v}\vec{v}^\transpose - tD)^{-1}\) are \(O(t^{-1})\) by part \eqref{lem:U-tA-pt2} of \Cref{lem:U-tA}, the entries of \(M(t)\) are \(O(t)\) as \(t \to 0\). By part \eqref{lem:invert-perturbed-pt1} of \Cref{lem:invert-perturbed}, there exists \(\epsilon_2 > 0\) such that \((I+M(t))\) is invertible for all \(0 < t < \epsilon_2\), with \((I+M(t))^{-1} = I + O(t)\) as \(t \to 0\). Taking \(\epsilon = \min\{\epsilon_1,\epsilon_2\}\) we see that \(\left(\vec{v}\vec{v}^\transpose - tD + B(t) \right) = (\vec{v}\vec{v}^\transpose - tD)(I + M(t))\) is invertible for all \(0 < t < \epsilon\); this proves part \eqref{pt:UL-inverse1} of the proposition. 

Now, using part \eqref{lem:U-tA-pt2} of \Cref{lem:U-tA}, we have
\begin{align*}
    \left(\vec{v}\vec{v}^\transpose - tD  + B(t)\right)^{-1} &= \left(I + M(t)\right)^{-1} (\vec{v}\vec{v}^\transpose - tD)^{-1} = \left(I + O(t)\right)  O(t^{-1}) = O(t^{-1}),
\end{align*}
which proves part \eqref{pt:UL-inverse2} of the proposition. Finally, using part \eqref{lem:U-tA-pt3} of \Cref{lem:U-tA}, we have
\begin{align*}
    \left(\vec{v}\vec{v}^\transpose - tD  + B(t)\right)^{-1} \vec{v}
    &= \left(I + M(t)\right)^{-1} (\vec{v}\vec{v}^\transpose - tD)^{-1} \vec{v}\\
    & = \left(I + O(t)\right) \left( \frac{1}{\vec{v}^\transpose D^{-1} \vec{v}} D^{-1}  \vec{v} + O(t) \right) = \frac{1}{\vec{v}^\transpose D^{-1} \vec{v}} D^{-1}\vec{v} + O(t),
\end{align*}
which proves part \eqref{pt:UL-inverse3} of the proposition.
\end{proof}


\subsection{Proof of Theorem \ref{thm:general-existence}}
\label{subsec:general-existence-proof}

We will now complete the proof of \Cref{thm:general-existence}, following the sketch we gave in \Cref{subsec:sufficient-conditions}. We are going to apply \Cref{prop:UL-inverse} in conjunction with the \demph{Schur complement inversion formula} \cite[Theorem~1.2]{Zhang}, whose statement we now recall. First, given vector spaces \(V_1\) and \(V_0\), for \(i \in \{0, 1\}\) let \(\dproj_i \colon V_1 \oplus V_0 \to V_i\) and \(\dinc_i \colon V_i \to V_1 \oplus V_0\) denote the projection and inclusion. A linear map  \(M \colon V_1 \oplus V_0 \to V_1 \oplus V_0\) can be written in block form as
\[
    	M = 
        \begin{bmatrix}
        	\dproj_1 M \dinc_1 & \dproj_1 M \dinc_0 \\
        	\dproj_0 M \dinc_1 & \dproj_0 M \dinc_0
    	\end{bmatrix} 
        \eqcolon
        \begin{bmatrix}
        	E_{11} & E_{10} \\
        	E_{01} & E_{00}
    	\end{bmatrix}.
\]
This gives component linear maps \(E_{ji} \colon V_i \to V_j\). If \(E_{11}\) is invertible, the \demph{Schur complement} of \(E_{11}\) in \(M\) is defined to be \(M/E_{11} \coloneq E_{00} - E_{01} E_{11}^{-1} E_{10}\). The Schur complement inversion formula says that \(M\) is invertible if and only if \(M/E_{11}\) is invertible, with 
\begin{equation}\label{eq:schur-complement}
    M^{-1} = 
    \begin{bmatrix} 
        E_{11}^{-1} + E_{11}^{-1} E_{10} (M/E_{11})^{-1} E_{01} E_{11}^{-1} & 
        -E_{11}^{-1} E_{10} (M/E_{11})^{-1} \\ 
        -(M/E_{11})^{-1} E_{01} E_{11}^{-1} & 
        (M/E_{11})^{-1} 
    \end{bmatrix}.
\end{equation}

Now, recall that a symmetric matrix \(D \in \R^{n\times n}\) determines an orthogonal decomposition of \(\R^n\) as \(\im(D) \oplus \ker(D)\). Denoting the projection and inclusion maps by \(\R^n \xrightarrow{\proj_1} \im(D) \xrightarrow{\inc_1} \R^n\) and \(\R^n \xrightarrow{\proj_0} \ker(D)\xrightarrow{\inc_0} \R^n\), 
the orthogonal splitting is represented by the inverse matrices
\[\changem \colon \R^n \leftrightarrows \im(D) \oplus \ker(D) \colon \changeminv.\]
Recall that \(D^\pseudo\) denotes the Moore--Penrose pseudoinverse of \(D\) (\Cref{def:moore-penrose-pseudo}). We saw in \Cref{lem:pseudo-magnitude} that if \(\matmag{D} \neq \infty\) then \(D^\pseudo\onevector\) is a weighting for \(D\) and we have \(\matmag{D} = \vec{v}^\transpose (D|_{\im(D)})^{-1} \vec{v}\) where \(\vec{v} = \proj_1 \onevector\).

\begin{thmgeneral-existence}
    Let \(X\) be a finite metric space whose distance matrix \(D\) satisfies \(\con{D} \neq 0\), and suppose that \((D^{\odot 2})|_{\ker(D)}\) is invertible. Then \(X\) admits a microscopic weighting if and only if \(\con{D} \neq \infty\). If so, then
    \[
	\lim_{t \to 0} \vec{w}(t) = \con{D} (\vec{u} - \vec{u}')
	\]
	where \(\vec{u} =  D^\pseudo \onevector\) and \(\vec{u}' = \inc_0 \left(D^{\odot 2}|_{\ker(D)} \right)^{-1} \proj_0 \left(D^{\odot 2}\right) \vec{u}\).
\end{thmgeneral-existence}

\begin{proof}
	If \(\con{D} = \infty\), then \Cref{prop:magD_zero_microweighting_does_not_exist} tells us that \(X\) does not admit a microscopic weighting; we will prove that if \(\con{D}\) is finite and non-zero, and \((D^{\odot 2})|_{\ker(D)}\) is invertible, then it does. We begin by considering the Hadamard power series
	\[ 
	\exp_{\odot}(-tD) = \onematrix - Dt + \frac{1}{2} D^{\odot 2} t^2 - \frac{1}{3!} D^{\odot 3} t^3 + \cdots.
	\]
	Now, we express this matrix in block form with respect to the decomposition \(\R^n \cong \im(D) \oplus \ker(D)\), as follows:
	\begin{align*}
	\changeminv \begin{bmatrix}
	\projm_1 \left(\onematrix - Dt + \frac{1}{2} D^{\odot 2} t^2 - O(t^3) \right) \incm_1 & \projm_1 \left(\onematrix - Dt + \frac{1}{2} D^{\odot 2} t^2 - O(t^3) \right) \incm_0  \\
	\projm_0 \left(\onematrix - Dt + \frac{1}{2} D^{\odot 2} t^2 - O(t^3) \right) \incm_1  & \projm_0 \left(\onematrix - Dt + \frac{1}{2} D^{\odot 2} t^2 - O(t^3) \right) \incm_0
	\end{bmatrix} \changem.
	\end{align*}
    Note that \(\projm_0 D = 0\) and \(D \incm_0 = 0\) by definition of \(\projm_0\) and \(\incm_0\). Also, the condition that \(\con{D} \neq 0\) is equivalent to \(\matmag{D} \neq \infty\), which says exactly that \(\onevector \in \im(D)\), so \(\projm_0 \onevector = \vec{0}\) and \(\onevector^\transpose \incm_0 = (\projm_0 \onevector)^\transpose = \vec{0}^\transpose\). It follows that \(\projm_0 \onematrix \incm_1 = \projm_1 \onematrix \incm_0 = \projm_0 \onematrix \incm_0 = 0\) as well, so in fact we have
    \begin{align*}
    	\exp_{\odot}(-tD)
    	& = 
        \changeminv 
        \begin{bmatrix}
            \projm_1 \onematrix \incm_1 - \projm_1 D \incm_1 t + O(t^2) &
            \frac{1}{2} \projm_1 D^{\odot 2} \incm_0 t^2 + O(t^3) \\
	        \frac{1}{2} \projm_0 D^{\odot 2} \incm_1 t^2 + O(t^3) & 
            \frac{1}{2} \projm_0 D^{\odot 2} \incm_0 t^2 + O(t^3)
        \end{bmatrix} 
        \changem\\
        &\eqcolon
        \changeminv 
        \begin{bmatrix}
        	E_{11}(t) & E_{10}(t) \\
        	E_{01}(t) & E_{00}(t)
    	\end{bmatrix}
        \changem.
	\end{align*}
	We know that for sufficiently small \(t > 0\) the matrix \(\exp_{\odot}(-tD)\) is invertible. We are going to compute \(\exp_{\odot}(-tD)^{-1} \onevector\) for such \(t\), making use of the Schur complement formula. First we will apply \Cref{prop:UL-inverse} to establish that the block \(E_{11}(t)\) is invertible for sufficiently small \(t > 0\), and to describe the vector \(E_{11}(t)^{-1} \vec{v}\), where \(\vec{v} = \projm_1 \onevector\). We will write down the Schur complement of \(E_{11}(t)\) and use \Cref{lem:invert-perturbed} to describe its inverse. Then we will invoke the Schur complement formula to describe the small-\(t\) asymptotics of \(\exp_{\odot}(-tD)^{-1} \onevector\).

    The linear map \(\projm_1 D \incm_1 = D|_{\im(D)}\) is invertible since \(\im (D) \cap \ker (D) = \{\zerovector\}\), and, by \Cref{lem:pseudo-magnitude}, we have
    \[\vec{v}^\transpose (D|_{\im(D)})^{-1} \vec{v} = \matmag{D} \neq 0,\]
    where \(\vec{v} = \projm_1 \onevector\). It therefore follows from part \eqref{pt:UL-inverse1} of \Cref{prop:UL-inverse} that \(E_{11}(t) = \vec{v} \vec{v}^\transpose - t D|_{\im(D)} + O(t^2)\) is invertible for all sufficiently small \(t > 0\). Part \eqref{pt:UL-inverse2} of the same proposition also tells us that \(E_{11}(t)^{-1}\) is \(O(t^{-1})\) as \(t \to 0\), and part \eqref{pt:UL-inverse3} tells us that
    \begin{equation}\label{eq:E_11-inverse}
        E_{11}(t)^{-1} \vec{v} = \frac{1}{\vec{v}^\transpose (D|_{\im(D)})^{-1} \vec{v}} \left(D|_{\im(D)}\right)^{-1} \vec{v} + O(t) = \frac{1}{\matmag{D}} \left(D|_{\im(D)}\right)^{-1} \vec{v} + O(t).
    \end{equation}
    The Schur complement of the block \(E_{11}(t)\) is the linear map 
	\[(E/E_{11})(t) \coloneq E_{00}(t) - E_{01}(t) E_{11}(t)^{-1}E_{10}(t).\]
	More explicitly, using the notation \(D^{\odot 2}|_{\ker(D)}\) for \(\projm_0 D^{\odot 2} \incm_0\), we can write
	\begin{equation}\label{eq:Schur-asymptotics}
	(E/E_{11})(t) = E_{00}(t) - O(t^2) O(t^{-1}) O(t^2) = \frac{t^2}{2} (D^{\odot 2}|_{\ker(D)}) + O(t^3) \quad \text{as } t \to 0.
	\end{equation}
    To see that \((E/E_{11})(t)\) is invertible, we use the assumption that \(D^{\odot 2}|_{\ker(D)}\) is invertible. That assumption lets us apply \Cref{lem:invert-perturbed} \eqref{lem:invert-perturbed-pt2} with \(A = D^{\odot 2}|_{\ker(D)}\) to see that
	\begin{equation}\label{eq:Schur-inverse-asymptotics}
    (E/E_{11})(t)^{-1} = \left(\frac{t^2}{2} \left(D^{\odot 2}|_{\ker(D)} + O(t) \right) \right)^{-1} = \frac{2}{t^2} \left(D^{\odot 2}|_{\ker(D)} \right)^{-1} + O(t^{-1}) \quad \text{as } t \to 0.
    \end{equation}
    In particular, \((E/E_{11})(t)^{-1} = O(t^{-2})\) as \(t \to 0\).
    
    We are now equipped to invoke the Schur complement inversion formula---but first we can make a further simplification. Since \(\onevector \in \im(D)\) by the assumption that \(\con{D} \neq 0\), we have \(\projm_0 \onevector = \vec{0}\) and hence \(\changem \onevector = \begin{bmatrix} \vec{v}^\transpose & \vec{0}^\transpose \end{bmatrix}^\transpose\). Thus, we can write
	\begin{align*}
	\exp_{\odot}(-tD)^{-1} \onevector &= \changeminv \begin{bmatrix}
	E_{11}(t) & E_{10}(t) \\
	E_{01}(t) & E_{00}(t)
	\end{bmatrix}^{-1} \changem \onevector  \\
	& = \changeminv \begin{bmatrix}
	B_{11}(t) & B_{10}(t) \\
	B_{01}(t) & B_{00}(t)
	\end{bmatrix} \begin{bmatrix} \vec{v} \\ \vec{0} \end{bmatrix} \\
	& = \changeminv \begin{bmatrix}
	B_{11}(t) \vec{v}\\
	B_{01}(t) \vec{v}
	\end{bmatrix} 
	\end{align*}
	where, by the Schur complement formula,
	\begin{equation}
    \label{eq:B11}
	    B_{11}(t) = E_{11}(t)^{-1} + E_{11}(t)^{-1} E_{10}(t) (E/E_{11})(t)^{-1} E_{01}(t) E_{11}(t)^{-1}
	\end{equation}
	and
	\begin{equation}
    \label{eq:B01}
	B_{01}(t) = - (E/E_{11})(t)^{-1} E_{01}(t) E_{11}(t)^{-1}.
	\end{equation}
	It only remains to compute \(B_{11}(t) \vec{v}\) and \(B_{01}(t) \vec{v}\).
	
    As \(t \to 0\) we have \(E_{11}(t)^{-1} = O(t^{-1})\) by \Cref{prop:UL-inverse} and \((E/E_{11})(t)^{-1} = O(t^{-2})\) by equation \eqref{eq:Schur-inverse-asymptotics}. Thus, from equations \eqref{eq:B11} and \eqref{eq:E_11-inverse}, we have
	\begin{align*}
	B_{11}(t) \vec{v} &= 
    \left( {\operatorname{id}_{\im(D)}} + E_{11}(t)^{-1} E_{10}(t) (E/E_{11})(t)^{-1} E_{01}(t) \right) E_{11}(t)^{-1} \vec{v} \\
    &= \left({\operatorname{id}_{\im(D)}} + O(t^{-1}) O(t^2) O(t^{-2}) O(t^2) \right) \left(\frac{1}{\matmag{D}} \left(D|_{\im(D)}\right)^{-1}\vec{v} + O(t)\right) \\
	& = \frac{1}{\matmag{D}} \left(D|_{\im(D)}\right)^{-1} \vec{v} + O(t) \quad \text{as } t \to 0.
	\end{align*}
	Meanwhile, it follows from equations \eqref{eq:B01}, \eqref{eq:Schur-inverse-asymptotics} and \eqref{eq:E_11-inverse} that
	\begin{align*}
	B_{01}(t) \vec{v} 
    &= \left(-\frac{2}{t^2} \left(D^{\odot 2}|_{\ker(D)} \right)^{-1} + O(t^{-1})\right) 
    \left(\frac{t^2}{2} \projm_0 D^{\odot 2} \incm_1 {+ O(t^3)}\right) 
    \left(\frac{1}{\matmag{D}} \left(D|_{\im(D)} \right)^{-1} \vec{v}  + O(t)\right) \\
	&= -\left(D^{\odot 2}|_{\ker(D)} \right)^{-1} \left(\projm_0 D^{\odot 2} \incm_1 \right) \frac{1}{\matmag{D}} \left(D|_{\im(D)}  \right)^{-1} \vec{v} + O(t) \\
	&= -\left(D^{\odot 2}|_{\ker(D)} \right)^{-1} \left(\projm_0 D^{\odot 2} \incm_1 \right) B_{11}(t) \vec{v} + O(t).
	\end{align*}
	Putting these computations together, as \(t \to 0\) we have
	\begin{align*}
	\exp_{\odot}(-tD)^{-1} \onevector &= \frac{1}{\matmag{D}} \changeminv \begin{bmatrix}
	\left(D|_{\im(D)}\right)^{-1} \vec{v}\\
	-\left(D^{\odot 2}|_{\ker(D)} \right)^{-1} \left(\projm_0 D^{\odot 2} \incm_1 \right) \left(D|_{\im(D)}\right)^{-1} \vec{v}
	\end{bmatrix}  + O(t) \\
	&= \con{D} \left(\incm_1 \left(D|_{\im(D)}\right)^{-1} \projm_1 \onevector - \incm_0 \left(D^{\odot 2}|_{\ker(D)}\right)^{-1} \left(\projm_0 D^{\odot 2} \incm_1 \right) \left(D|_{\im(D)}\right)^{-1} \projm_1 \onevector \right) + O(t)\\
	&= \con{D} \left(\vec{u}- \incm_0 \left( D^{\odot 2}|_{\ker(D)} \right)^{-1} \left(\projm_0 D^{\odot 2} \right) \vec{u} \right) + O(t) \\
	&= \con{D} \left( \vec{u} - \vec{u}'\right) + O(t),
	\end{align*}
    where the third line uses the fact that \(\incm_1 \left(D|_{\im(D)}\right)^{-1} \projm_1 = D^\pseudo\) by \Cref{def:moore-penrose-pseudo}. 
	Taking the limit as \(t \to 0\) gives the result.
\end{proof}


\bibliographystyle{amsplain}
\bibliography{RW_microweight.bib}

\end{document}